\documentclass[11pt,a4paper,twoside]{amsart}
\usepackage{fancyhdr}
\usepackage{pdflscape}
\usepackage{color}
\usepackage{longtable}
\usepackage{multirow}
\usepackage[pdftex]{graphicx}
\usepackage{array}
\usepackage{tikz}

\usepackage[utf8]{inputenc}
\usepackage[T1]{fontenc}

\usepackage{amsmath,amsfonts,amsthm,amsopn,color,amssymb, latexsym}
\usepackage{enumitem}
\usepackage{palatino}
\usepackage{graphicx}
\usepackage[colorlinks=true]{hyperref}
\hypersetup{urlcolor=blue, citecolor=red, linkcolor=blue}
\usepackage{relsize}
\usepackage{esint}

\usepackage{cite}

\def\Xint#1{\mathchoice
  {\XXint\displaystyle\textstyle{#1}}%
  {\XXint\textstyle\scriptstyle{#1}}%
  {\XXint\scriptstyle\scriptscriptstyle{#1}}%
  {\XXint\scriptscriptstyle\scriptscriptstyle{#1}}%
  \!\int}
\def\XXint#1#2#3{{\setbox0=\hbox{$#1{#2#3}{\int}$}
  \vcenter{\hbox{$#2#3$}}\kern-.5\wd0}}
\def\fint{\Xint-}

\newcommand{\al}{\alpha}

\newcommand{\dt}{\delta}
\newcommand{\eps}{\varepsilon}
\newcommand{\HTI}{\tilde{H}}

\newcommand{\lm}{\lambda}

\newcommand{\e}{\varepsilon}

\newcommand{\R}{\mathbb{R}}
\newcommand{\Z}{\mathbb{Z}}
\newcommand{\Sd}{\mathbb{S}^2}

 \DeclareMathOperator{\Id}{Id}

\renewcommand{\a }{\alpha }

\renewcommand{\l }{\lambda}

\renewcommand{\O}{\Omega}
\renewcommand{\OE}{\Omega_\varepsilon}

\renewcommand{\S}{\Sigma}

\newcommand{\N}{\mathbb{N}}

\def\bbm[#1]{\mbox{\boldmath $#1$}}
\newcommand{\beq }{\begin{equation}}
\newcommand{\eeq }{\end{equation}}

\newtheorem{theorem}{Theorem}[section]
\newtheorem{lemma}[theorem]{Lemma}

\newtheorem{proposition}[theorem]{Proposition}
\newtheorem{remark}[theorem]{Remark}

\def\sideremark#1{\ifvmode\leavevmode\fi\vadjust{\vbox to0pt{\vss
 \hbox to 0pt{\hskip\hsize\hskip1em
 \vbox{\hsize3cm\tiny\raggedright\pretolerance10000
  \noindent #1\hfill}\hss}\vbox to8pt{\vfil}\vss}}}%

\newcommand{\mybinom}[2]{\Bigl(\begin{array}{@{}c@{}}#1\\#2\end{array}\Bigr)} 

\usepackage{epsfig,color,xcolor}

\title[The singular mean field problem with sign-changing potentials]{Compactness, existence and multiplicity for the singular mean field problem with sign-changing potentials}

\author{ Francesca de Marchis, Rafael L\'{o}pez-Soriano and David Ruiz}
\address{Francesca De Marchis, Dipartimento di Matematica, Universit\`{a} di Roma {\em Sapienza}, P.le Aldo Moro 5, 00185 Roma, Italy.}
\email{demarchis@mat.uniroma1.it}
\address{Rafael L\'{o}pez-Soriano, Departamento de An\'{a}lisis Matem\'{a}tico, Universidad de Granada, Campus Fuentenueva, 18071 Granada, Spain.}
\email{rafals@ugr.es}
\address{David Ruiz, Departamento de An\'{a}lisis Matem\'{a}tico, Universidad de Granada, Campus Fuentenueva, 18071 Granada, Spain.}
\email{daruiz@ugr.es}
\thanks{}

\email{}
\date{}



\begin{document}

\begin{abstract} In this paper we consider a mean field problem on a compact surface without boundary in presence of conical singularities. The corresponding equation, named after Liouville, appears in the Gaussian curvature prescription problem in Geometry, and also in the Electroweak Theory and in the abelian Chern-Simons-Higgs model in Physics. Our contribution focuses on the case of sign-changing potentials, and gives results on compactness, existence and multiplicity of solutions.

\end{abstract}

\maketitle

\section{Introduction}

The classical problem of prescribing the Gaussian curvature on a compact surface $\Sigma$ under a conformal change of the metric dates back to \cite{Ber, KazWar}. Let us denote by $g$ the original metric, $\tilde{g}$ the new one and $e^v$ the conformal factor (that is, $\tilde{g} = e^v g$). The problem reduces to solving the PDE

$$ - \Delta_g v + 2 K_g(x) = 2 K(x)e^{v},$$
where $K_g$, $K$ denote the curvature with respect to $g$ and $\tilde{g}$, respectively. Observe that by the Uniformization Theorem we can assume that $K_g$ is a constant. The solvability of this equation has been studied for a long time, and it is not possible to give here a comprehensive list of references.  

The above setting needs to be modified if one wants to prescribe also the appearance of conical singularities on the surface, a case that was first studied in \cite{Troy}. We recall that a conformal metric $\tilde{g}$ has a conical singularity at $p$ of order $\alpha \in (-1,+\infty)$ if there exist local coordinates $z$ in $\mathbb{C}$ such that $z(p)=0$ and 

\beq
\label{conical}\hat{g}(z) = e^{\psi} |z|^{2 \alpha} |dz|^2,
\eeq
where $\hat{g}$ is the local expression of $g$ and $\psi$ is continuous in a neighborhood of $0$ in $\mathbb{C}$ and $C^2$ outside the point $p$.

\bigskip 
In this case we are led with weak solutions of the problem

\beq \label{sing}
- \Delta_g v + 2 K_g= 2 K(x)e^{v} - 4 \pi \sum_{j=1}^m \alpha_j \delta_{p_j},
\eeq
where $\delta_{p_j}$ denotes a Dirac delta at the point $p_j \in \Sigma$ (see Appendix of \cite{BarDemMal} for a rigorous deduction of \eqref{sing}). Integrating the above equation and taking into account the Gauss-Bonnet formula, we obtain

\beq \label{bonnet0}
4 \pi \chi(\Sigma) = 2 \int_{\Sigma} K e^v \, dV_g - 4 \pi \sum_{j=1}^m \alpha_j.
\eeq

%

We now transform equation \eqref{sing} into another one which admits a variational structure. Let $G(x,y)$ be the Green function of the Laplace-Beltrami operator on $\Sigma$ associated to $g$, i.e.

\beq\label{fgreen}
-\Delta_g G(x,y)=\dt_y-\frac{1}{|\Sigma|}\quad \mbox{in}\quad \Sigma,\qquad\quad\int_{\Sigma }G(x,y)dV_{g}(x)=0.
\eeq

We define

\beq\label{accame}
h_m(x)=4\pi\sum\limits_{j=1}^m\al_j G(x,p_j)=2\sum\limits_{j=1}^m \al_j\log \left (\frac{1}{d(x,p_j)} \right )+ 2 \pi \alpha_j H(x, p_j),
\eeq
where $H$ is the regular part of $G$. By the change of variable
$$
u= v + h_m
$$
we can pass to the equation
\beq\label{equation}\tag*{$(*)_\lambda$}
-\Delta_g u=\lm\left(\frac{\tilde K e^{u}}{\int_{\Sigma}\tilde K e^{u}dV_{g}}-\frac1{|\Sigma|}\right)\quad\mbox{in}\quad \Sigma,
\eeq
where
\begin{equation}\label{tildeK}
\tilde K=K e^{-h_m},
\end{equation}
and, according to \eqref{bonnet0}, $\lambda$ is given by

\beq \label{bonnet} \lambda = 4\pi(\chi(\Sigma) + \sum_{j=1}^m\alpha_j). \end{equation}

Notice that $e^v = e^{-h_m} e^{u}$ which is consistent with \eqref{conical}. Observe also that

$$
\tilde{K}(x) \simeq d(x,p_j)^{2\alpha_j} K(x) \quad \mbox{\quad close to $p_j$.}
$$

\medskip This equation appears also in Physics, in the mathematical  Glashow-Salam-Weinberg model of the Electroweak Theory and in the abelian Chern-Simons-Higgs model, see \cite{Dun, Lai, Tarbook, yang}. In this context, the points $p_j$ represent vortices of order $\alpha_j \in \N$. There are by now many works dealing with this problem, see \cite{BarDemMal, BarMal, BarLin, BarLin2, BarLinT, bt, btnoradial, Car1, CarMal, ChenLin, ChenLin2, DDI, DeMLS, kuolin, MalchiodiRuizSphere, Mon, Tarhb}. In this framework, the equation receives the name of (singular) mean field equation. We highlight that under this perspective the restriction \eqref{bonnet} is not present.

\medskip In this paper we are concerned with equation \ref{equation} in the case in which $K$ is a sign-changing function. We give existence and generic multiplicity results by means of variational methods. For that one also needs to show that solutions of \ref{equation} are a priori bounded, which in this case means that they form a compact set.

\medskip Regarding compactness of solutions, one can pose the question as follows: given $u_n$ a sequence
of solutions of \ref{equation} for  $\lambda=\lambda_n \to \lambda_0$, is it uniformly bounded? This question has been addressed in \cite{breme, lisha} for the regular problem, and in \cite{barmon, bt} for the equation with vortices, always for positive potentials $K(x)$. In summary, if blow-up occurs then $e^{u_n}$ concentrate around a finite set of critical points, and a quantization argument shows that $\lambda_0$ must belong to a certain discrete critical set. Here, the assumption on the positivity of $K$ is not just a technical one, as can be inferred from some recent examples of blowing-up solutions in \cite{BGS, DePR}. Those solutions concentrate around local maxima of $K$ at $0$ level.

The question of compactness has been elided in \cite{RuizSoriano} by using energy estimates in a related problem posed on a surface with boundary. This technique is however very much restricted to the case considered there.

For sign-changing functions $K$ the first related compactness result is \cite{ChenLiAnn}. That paper is concerned with the scalar curvature prescription problem, a higher dimension analogue of our problem which has also attracted much attention in the literature. The authors are able to show compactness of solutions under the hypothesis

\begin{enumerate}[label=(H1), ref=(H1)]
\item \label{H1} $K$ is a sign-changing $C^{2,\alpha}$ function with $\nabla K (x) \neq 0$ for any $x \in \Sigma$ with $K(x)=0$.
\end{enumerate}

An improvement of this technique has been given in \cite{ChenLiMPMS}. The general idea is to first derive uniform integral estimates, which allow one to obtain a priori estimates in the region $\{ x \in \Sigma:\ K(x)<-\delta\}$, for $\delta>0$ small. Then the moving plane technique is used to compare the values of $u$ on both sides of the nodal curve $\Gamma=\{x \in \Sigma: \ K(x)=0\}$. This, together with the aforementioned integral estimate, implies boundedness in a neighborhood of $\Gamma$. Finally, we rely on \cite{lisha, bt, barmon} for the region $ \{ x \in \Sigma:\ K(x)> \delta \}$.

The approach of \cite{ChenLiAnn} has been partially adapted to problem \ref{equation} in \cite{ChenLiDCDS, DeMLS}. However, those results use the stereographic projection to pass to a global problem in the plane and are hence restricted to $\Sigma = \Sd$. Moreover, the derivation of the integral estimate \cite[Lemma 2.2]{ChenLiDCDS}, essential in both papers, is not completely clear. One of the goals of this paper is to settle the question of compactness: we show compactness for \ref{equation} in any compact surface under assumption \ref{H1}.

Our approach follows the ideas of \cite{ChenLiMPMS}. The main difficulty with respect to \cite{ChenLiMPMS} comes from the fact that $u_n$ is neither positive nor uniformly bounded from below, a priori. This is a problem for the integral estimate in \cite{ChenLiDCDS, ChenLiMPMS}, and also for the use of the moving plane method near the nodal curve. In our proofs we first estimate the negative part of $u_n$ by using Kato inequality. This is the key for the proof of the integral estimate and is also essential to perform the comparison argument by the moving plane method.

For what concerns existence and multiplicity of solutions, we shall restrict ourselves to the case of positive orders $\alpha_j$. Our proofs make use of variational methods. Indeed, problem \ref{equation} is the Euler-Lagrange equation of the energy functional

\beq\label{functional}
I_\lambda(u)=\frac12\int_{\S} |\nabla u|^2 dV_{g}+\frac\lambda{|\S|}\int_{\S} u \, dV_{g}-\lambda \log\int_{\S} \tilde K e^u dV_{g},
\eeq
defined in the domain

\beq\label{X}
X=\left\{u\in H^1(\S):  \ \int_{\S} \tilde K e^u\,dV_{g}>0\right\}.
\eeq
If $\lambda < 8\pi$, then $I_\lambda$ is coercive and a minimizer exists, see \cite{Troy}. Instead, $I_\lambda$ is not bounded from below if $\lambda > 8 \pi$. This range is the main concern of this paper, and we shall use a Morse theoretical approach to find solutions of \ref{equation} which are saddle-type critical points of $I_\lambda$. In order to do that we study the topology of the energy sublevels of $I_\lambda$. In a certain sense, a function $u$ at a low energy level concentrate around a certain number of points of $\Sigma$, and the topology of the space of those configurations plays a crucial role in the min-max argument. This approach has been followed in different works. For instance, \cite{DjadliMalchiodi} considers an analogue problem in dimension 4, whose ideas can be applied equally well to the regular mean field problem (see \cite{Djadli}), whereas \cite{BarDemMal, CarMal, MalchiodiRuizSphere} deal with the singular problem. Let us point out that all those papers consider positive functions $K$. For sign changing potentials, the location of the points is restricted to the set $\overline{\{x \in \Sigma: \ K(x)>0\}}$, and this fact changes dramatically the topology of the space of configurations.

Roughly speaking, to obtain existence of solutions one needs to show that the very low energy sublevels form a non-contractible set. If $\Sigma= \Sd$ this study has been carried out in \cite{DeMLS}, yielding existence of solutions. In this paper we consider the case of general $\Sigma$ and we also prove multiplicity results. For multiplicity, we need to estimate the sum of the Betti numbers of the energy sublevels. This multiplicity result is valid under nondegeneracy assumptions, which are generic (in the couple $(K,g)$) by a transversality argument.

In general, min-max arguments yield existence of solutions provided that the well-known Palais-Smale property is satisfied. In this type of problems the validity of that property is still an open problem; however we can circumvent this difficulty by using the deformation lemma established in \cite{Lucia}, in the spirit of \cite{Struwe}. This technique has now become well known, and relies on compactness of solutions of approximating problems. At this point our aforementioned compactness result comes into play.


The rest of the paper is organized as follows. In Section~\ref{sectnotrespre} we set the notation, recall some preliminary results and state the main theorems proved in the paper. Section~\ref{sectcompac} is devoted to the proof of the compactness result, see Theorem~\ref{thmcompa}. In Section~\ref{sectproj} we give a description of the topology of the energy sublevels. The estimation of the dimension of the homology groups of those sets is made in Section~\ref{sectcomphomol}. Finally, the proofs of our main results are completed in Section~\ref{sectproof}.

\section{Notations, main results and preliminaries}\label{sectnotrespre}

\setcounter{equation}{0}

In this section we fix the notation used in this paper, formulate the principal results obtained and collect some preliminary known results.

From now on $(\S,g)$ will be a compact surface without boundary $\Sigma$ equipped with a Riemannian metric $g$ and $d(x,y)$ will denote the distance between two points $x, y \in \S$ induced by the ambient metric. $B_p(r)$ stands for the open ball of radius $r>0$ and center $p \in \Sigma$ and

\[
\Omega^r=\{x\in \S: \ d(x,\Omega)<r\}.
\]
Given $f \in L^1(\S)$, we denote the mean value of $f$ by $\fint_{\S} f=\frac1{|\S|} \int_{\S} f$, where $|\S|$ is the area of $\S$. \\
Since the functional $I_\lambda$ is invariant under addition of constants, we can restrict its domain to functions with $0$ mean. In other words, we can consider $I_\lambda$ defined in $\bar X$, where

\beq\label{Xbar}
\bar X=\{u\in X: \ \int_\Sigma u \,dV_g=0\}.
\end{equation}
For a real number $a$, we introduce the following notation for the sublevels of the energy functional $I_\lambda$ (defined in \eqref{functional}) restricted to $\bar X$
$$
I_{\lambda}^{a}= \{u\in \bar
X:\ I_\lambda(u)\leq a\}.
$$

The symbol $\amalg$ will be employed to denote the disjoint union of sets.

Throughout the paper, the sign $\simeq$ refers to homotopy equivalences, while $\cong$ refers to homeomorphisms between topological spaces or isomorphisms between groups.

Given a metric space $M$ and $k\in \mathbb{N}$, we denote by $Bar_k(M)$ the set of formal barycenters of order $k$ on $M$, namely the following family of unit measures supported in at most $k$ points

\beq\label{formbar}
Bar_k(M)=\left\{\sum_{i=1}^k t_i \delta_{x_i}:\ t_i\in[0,1],\ \sum_{i=1}^k t_i =1,\ x_i \in M \right\}.
\eeq
We consider $Bar_k(M)$ as a topological space with the weak$^*$ topology of measures.\\

Uninfluential positive constants are denoted by $C$, and the value of $C$ is allowed to vary from formula to formula.

\subsection{Main results}

\

Let us define the sets
$$
\S^{+}=\{x\in \S: \ K(x) > 0\}, \qquad \S^{-}=\{x\in \S: \ K(x) < 0\}, \qquad \Gamma=\{x\in \S: \ K(x)=0\}.
$$

Assumption \ref{H1} implies that the nodal line $\Gamma$ is regular and that
\beq\label{Npm}
N^{+}=\#\{\textnormal{connected components of $\S^+$}\}<+\infty.
\eeq

In what follows we will assume that:

\

\begin{enumerate}[label=(H2), ref=(H2)]
\item \label{H2} $p_j\notin \Gamma$ \ for all $j\in\{1,\ldots,m\}$.
\end{enumerate}

\

So we can suppose, up to reordering, that there exists $\ell\in\{0,\ldots,m\}$ such that
\beq\label{ell}
p_j\in{\S^+}\textrm{ for $j\in\{1\,\ldots,\ell\}$,}\quad p_j\in\S^- \textrm{\; for $j\in\{\ell+1,\ldots,m\}$.}
\eeq

\

For $K>0$ it is known that the blow-up phenomena can only occur if the parameter $\lambda$ takes the form $8\pi r + \sum_{j=1}^{m}8\pi(1+\alpha_j)n_j$, with $r\in \mathbb{N}\cup \{0\}, n_j\in\{0,1\}$ (see \cite{bt}). Therefore the set of solutions is compact if $\lambda$ has a different form. In the next theorem we obtain an analogous conclusion without the sign restriction, where the set of critical values is
\beq\label{crit}
\Lambda=\left\{8\pi r + \sum_{j=1}^{\ell}8\pi(1+\alpha_j)n_j: \ r\in \mathbb{N}\cup \{0\}, n_j\in\{0,1\}\right\} \setminus \{0\}.
\eeq

\begin{theorem}\label{thmcompa} Assume that $\alpha_1,\ldots,\alpha_m>-1$ and let $K_n$ be a sequence of functions with $K_n \to K$ in $C^{2,\alpha}$ sense, where $K$ satisfies \ref{H1}, \ref{H2}. Let  $u_n$ be a sequence of weak solutions of the problem

\beq\label{equationbis}
-\Delta_g u_n= \tilde K_n e^{u_n}-f_n \quad \mbox{in}\quad \Sigma,
\eeq
with $f_n\to f$ in $C^{0,\alpha}$ sense and $\tilde{K}_n=K_n e^{-h_m}$ with $h_m$ given by \eqref{accame}. Then, the following alternative holds

\begin{enumerate}
\item either $u_n$ is uniformly bounded in $L^{\infty}(\Sigma)$;

\item or, up to a subsequence, $u_n$ diverges to $-\infty$ uniformly;

\item or, up to a subsequence, $ \lim_{n\to +\infty} \int_\Sigma \tilde{K}_n e^{u_n} \in \Lambda$, defined in \eqref{crit}.
\end{enumerate}

\end{theorem}

We point out that that the conical singularities located in $\S^-$ do not play any role in the compactness result. Observe also that equation \ref{equation} can be written in the form \eqref{equationbis} by adding a suitable constant to $u_n=u$, if $\tilde K_n=\tilde K$ and $f_n = \frac{\lambda}{|\Sigma|}$.

\

In order to state our existence result we introduce an additional assumption on $K$:

\begin{enumerate}[label=(H3), ref=(H3)]
\item \label{H3} $N^+>k$ or $\S^+$ has a connected component which is not simply connected,
\end{enumerate}
where $N^+$ is defined in \eqref{Npm}.

\begin{theorem}\label{thm:existencegeneral}
Let $\alpha_1,\ldots,\alpha_\ell>0$, where $\ell$ is defined in \eqref{ell}, and $\lambda\in(8 k \pi, 8 (k+1)\pi) \setminus\Lambda$. If \ref{H1}, \ref{H2}, \ref{H3} are satisfied then \ref{equation} admits a solution.
\end{theorem}

For $K>0$, and thus $\Sigma^+=\Sigma$ and $N^+=1$, \ref{H3} is satisfied if the surface $\Sigma$ has positive genus; this case has been covered in \cite{BarDemMal}.

If $\Sigma^+$ has trivial topology Theorem \ref{thm:existencegeneral} is not applicable. We can give a result also in this case, following the ideas of \cite{MalchiodiRuizSphere}. For that, we define the set:

\beq\label{Jlambda}
J_\lambda=\{p_j\in \S^+: \ \lambda<8\pi(1+\alpha_j)\}
\eeq
and we introduce the hypothesis

\begin{enumerate}[label=(H4), ref=(H4)]
\item \label{H4} $J_{\lambda}\neq\emptyset$.
\end{enumerate}

\begin{theorem}\label{thm:existence8pi16pi}
Let $\alpha_1,\ldots,\alpha_\ell\in(0,1]$, where $\ell$ is defined in \eqref{ell}, and $\lambda\in(8\pi,16\pi)\setminus\Lambda$. If \ref{H1}, \ref{H2}, \ref{H4} are satisfied then \ref{equation} admits a solution.
\end{theorem}

\begin{remark} There are many examples of applications of these results to the geometric problem commented in the Introduction. Just to exhibit an example, let us consider the problem of prescribing a conformal metric in $\Sigma = \mathbb{T}^2$  with gaussian curvature $K$ and one conical point $p$ of order $\alpha$. Assume that assumptions \ref{H1}, \ref{H2} are satisfied. Then Theorem \ref{thm:existencegeneral} implies that the problem is solvable if one of the following assumptions are satisfied:
\begin{enumerate}
\item $\alpha \in (2k, 2k+2)$ with $k \in \N$ and $\Sigma^+$ has more than $k$ connected components.
\item $\alpha \in (2k, 2k+2)$ with $k \in \N$ and $\Sigma^+$ has a component which is not simply connected.
\end{enumerate}
Let us now consider the same problem but with $ m $ conical points, all of them of order $\alpha$. Then Theorem \ref{thm:existence8pi16pi} implies that the geometric problem is solvable if $\alpha\in(0,1]$ and $2< m \, \alpha < 2 + 2\alpha$ and at least one conical point is placed in $\Sigma^+$.

Many other examples can be constructed.

\end{remark}

\bigskip

In order to now state our multiplicity results we introduce some more notation. Let us denote by $A_i$ the non-contractible connected components of $\S^+$ and $C_h$ be the contractible ones, $i=1,\ldots,N$, $h=1,\ldots,M$ and $N, M\in\N\cup\{0\}$, $N+M=N^+$. Obviously,

\[
\Sigma^+=\coprod_{i=1}^N A_i \amalg \coprod_{h=1}^M C_h.
\]

Recall that a \textit{bouquet of $g$ circles} is a set $B^g=\bigcup_{j=1}^g S'_j$ where $S'_j$ are simple closed curves verifying that $S'_i \cap S'_j=\{q\}$. If $A_i$ has genus $\textit{g}_i$ and $\textit{b}_i$ boundary components, it is well known that $A_i$ can be retracted to an inner bouquet $B^{{g}_i}$, where $g_i=2\textit{g}_i+\textit{b}_i-1$. Instead, $C_h$ is homotopically equivalent to any point $y_h\in C_h$. Therefore

\beq\label{concomp}
\Sigma^+\simeq \coprod_{i=1}^N B^{{g}_i}\,\amalg\,\{y_1,\ldots,y_M\}, \quad \mbox{ with } \quad g_i=2\textit{g}_i+\textit{b}_i-1 \quad \mbox{ for \quad $i=1,\ldots,N$.}
\end{equation}

Define $\mathcal{M}$ as the space of all Riemannian metrics on $\Sigma$ equipped with the $\mathcal{C}^{2,\alpha}$ norm and

\beq\label{K_ell}
\mathcal{K}_{\ell}=\left\lbrace K:\Sigma\to\R: \begin{array}{l}
\mbox{$K$ satisfies \ref{H1}, \ref{H2}} \\
\mbox{$p_1,\ldots,p_\ell\in\Sigma^+,\;p_{\ell+1},\ldots,p_m\in\Sigma^-$}
\end{array}\right\rbrace ,
\end{equation}
also equipped with the $C^{2,\alpha}$ norm.

\begin{theorem}\label{thm:multiplicitygeneral}
Let $\ell\in\{0,\ldots,m\}$ and let us assume that $\alpha_1,\ldots,\alpha_\ell>0$. If $\lambda\in(8k\pi,8(k+1)\pi)\setminus\Lambda$, $k\in\N$, then for a generic choice of function $K$ and metric $g$ (namely for $(K,g)$ in an open and dense subset of $\mathcal{K}_{\ell}\times \mathcal M$),
then
\[
\#\{\mbox{solutions of \ref{equation}}\}\geq \sum_{q\geq0} d_q,
\]
where if \quad $k+1-M \leq N$, then

\[
d_q=\left\{
\begin{array}{ll}
\mybinom{N+M-1}{N+M-p}\mathlarger{\sum}\limits_{\tiny\begin{array}{c}a_1+\ldots+a_N=k-p+1\\a_i\geq0\end{array}}\!\!\!\!\!\!\!\!\!\!\!\!\!\!\!\!\!\!\!s_{a_1,g_1}\ldots s_{a_N,g_N} &\mbox{\footnotesize \quad if\;\, $q=2k-p\;\;(1\leq p\leq k+1),$}\\
0 & \mbox{\quad\footnotesize otherwise;}
\end{array}\right.
\]
while if \quad $k+1-M \geq N$, then

\[
d_q=\left\{
\begin{array}{ll}
\mybinom{N+M-1}{N+M-p}\mathlarger{\sum}\limits_{\tiny\begin{array}{c}a_1+\ldots+a_N=k-p+1\\a_i\geq0\end{array}}\!\!\!\!\!\!\!\!\!\!\!\!\!\!\!\!\!\!\!s_{a_1,g_1}\ldots s_{a_N,g_N} &\mbox{\footnotesize \quad if\;\, $q=2k-p\;\;(1\leq p\leq N),$}\\
\mybinom{N+M-s}{M-s}\mathlarger{\sum}\limits_{\tiny\begin{array}{c}a_1+\ldots+a_N=k-N-s+1\\a_i\geq0\end{array}}\!\!\!\!\!\!\!\!\!\!\!\!\!\!\!\!\!\!\!s_{a_1,g_1}\ldots s_{a_N,g_N} &\mbox{\footnotesize \quad if\;\, $q=2k-N-s\;\;(1\leq s\leq M),$}\\
0 & \mbox{\quad\footnotesize otherwise;}
\end{array}\right.
\]
with $s_{a,g}=\binom{a+g-1}{g-1}$ and $g_i$ defined in \eqref{concomp}. \\

Moreover we adopt the following convention: if $N=0$

\[
\mathlarger{\sum}\limits_{\tiny\begin{array}{c}a_1+\ldots+a_N=h\\a_i\geq0\end{array}}\!\!\!\!\!\!\!\!\!\!\!\!\!s_{a_1,g_1}\ldots s_{a_N,g_N}=\left\{\begin{array}{ll}
1\quad&\mbox{if $h=0$,}\\
0&\mbox{if $h\neq 0$.}
\end{array}
\right.
\]\\
Notice that if $k+1-M=N$ the two formulas coincide.

\end{theorem}

\

We point out that $ \#\{\mbox{solutions of \ref{equation}}\} \to + \infty$ as $N^+=N+M \to + \infty$.

\

The above result gives no information if $\Sigma^+$ has trivial topology; however, our second multiplicity result can be applied also in this case.

\begin{theorem}\label{thm:multiplicity8pi16pi}
Let $\ell\in\{0,\ldots,m\}$ and let us assume that $\alpha_1,\ldots,\alpha_\ell\in(0,1]$. If $\lambda\in(8\pi,16\pi)\setminus \Lambda$, then for a generic choice of function $K$  and metric $g$ (namely for $(K,g)$ in an open and dense set of $\mathcal{K}_{\ell}\times \mathcal M$), then

\[
\#\{\mbox{solutions of \ref{equation}}\}\geq N^+ -1+\sum_{i=1}^N g_i+|J_\lambda|,
\]
where the set $J_\lambda$ is defined in \eqref{Jlambda} and $g_i$ in \eqref{concomp}.
\end{theorem}

\begin{remark}
We point out that by standard elliptic estimates any solution $u$ of \ref{equation} is classical if all the $\alpha_j$'s are positive, while if $\alpha_j\in(-1,0)$ for $j\in J\subset\{1,\ldots,m\}$, then $u\in C^{2}(\Sigma\setminus\{\cup_{j\in J} p_j\})\cap C^{0}(\Sigma)$.
\end{remark}

\subsection{Topological and Morse-theoretical preliminaries}\label{ss:topprel}

\

In this subsection we recall a classical theorem on Morse inequalities. Furthermore we give a short review of basic notions of algebraic topology needed to get the multiplicity estimates. Finally, we state a recent result concerning the topology of barycenter sets with disconnected base space.

Given a pair of spaces
$(X,A)$ we will denote by $H_q(X,A;\Z_2)$ the relative q-th homology group with coefficient
in $\Z_2$ and by $\HTI_q(X;\Z_2)=H_q(X,x_0;\Z_2)$ the reduced homology with coefficient in $\Z_2$, where $x_0\in X$. We adopt the convention that $\HTI_q(X;\Z_2)=0$ for any $q<0$.\\
Finally, if $X$, $Y$, are two topological spaces and $f:X\to Y$ is a continuous function, we will denote by $f_*:H_q(X;\Z_2)\to H_q(Y;\Z_2)$ the pushforward morphism induced by $f$.

\

Let us first recall a classical result in Morse theory.
\begin{theorem}\label{MorseIneq} (see e.g. \cite{Chang}, Theorem
4.3) Suppose that $H$ is a Hilbert manifold, $I\in C^2(H;\R)$ satisfies
the $(PS)$-condition at any level $c\in[a,b]$, where $a$, $b$ are regular values for $I$. If all the critical points of $I$ in $\{a\leq I\leq b\}$ are nondegenerate, then

\[
\#\{\textrm{\small{critical points of $I$ in $\{a\leq I\leq b\}$ with index $q$}}\}\geq\dim(H_q(\{I\leq b\},\{I\leq a\};\Z_2))\qquad\quad \mbox{for any $q\geq0$,}
\]
where we call (Morse) index of $u\in H$ the number of negative eigenvalues (counted with multiplicity) of the selfadjoint operator $d^2 I(u)$.
\end{theorem}

\

In what follows we collect some well-known definitions and results in
algebraic topology and we refer to \cite{Hatcher} for further details.

\

\noindent\textbf{Wedge sum.} Given spaces $C$ and $D$ with chosen points $c_0\in C$ and $d_0\in D$, then the wedge sum $C\vee D$ is the quotient of the disjoint union $C\amalg D$  obtained by identifying $c_0$ and $d_0$ to a single point.  If $\{c_0\}$ (resp. $\{d_0\}$) is a closed subspace of  $C$ (resp. $D$) and is a deformation retract of some neighborhood in $C$ (resp. $D$), then

\beq\label{formula homology wedge}
 \HTI_q(C\vee D;\Z_2)\cong \HTI_q(C;\Z_2)\bigoplus\HTI_q(D;\Z_2),
  \end{equation}
see \cite[Corollary 2.25]{Hatcher}.\\

\noindent\textbf{Smash Product.} Inside a product space $X\times Y$ there are copies
of $X$ and $Y$, namely $X\times\{y_0\}$ and $\{x_0\}\times Y$ for points
$x_0\in X$ and $y_0\in Y$. These two copies of $X$ and $Y$ in $X\times Y$
intersect only at the point $(x_0,y_0)$, so their union can be identified
with the wedge sum $X\vee Y$. The smash product $X\wedge Y$ is then defined
to be the quotient $X\times Y/X\vee Y$. For the reduced homology of the smash product the following formula holds, \cite[page 276]{Hatcher},

\beq\label{formula homology smash}
\HTI_{q}(C\wedge D;\Z_2)\cong\bigoplus_{i+j=q-1}(\HTI_i(C;\Z_2)\otimes \HTI_j(D;\Z_2)).
\end{equation}

\

\noindent\textbf{Unreduced suspension.} The unreduced suspension (often, as in \cite{Hatcher}, denoted by $SC$) is defined to be
\beq
\Sigma C=(C\times [0,1])/\{(c_1,0)\simeq(c_2,0) \mbox{ and } (c_1,1)\simeq(c_2,1) \mbox{ for all }c_1,c_2\in C \}.
\end{equation}
For the reduced homology of the unreduced suspension the following formula holds, \cite[page 132, ex. 20]{Hatcher},

\beq\label{formula homology suspension}
\HTI_{q+1}(\Sigma C;\Z_2)\cong \HTI_q(C;\Z_2).
\end{equation}
%

\

\noindent\textbf{Join.} The join of two spaces $C$ and $D$ is the
space of all segments joining points in $C$ to points in $D$.
It is denoted by $C\ast D$ and is the identification space
$$
C\ast D= C\times [0,1]\times D/(c,0,d)\sim (c',0,d), (c,1,d)\sim (c,1,d')\qquad\forall\,c,\,c'\in C, \forall\, d,\,d'\in D.
$$
Being $C\ast D\simeq \Sigma(C\vee D)$, \cite[page 20, ex. 24]{Hatcher}, we have that

\beq\label{formula homology join}
\HTI_q(C\ast D;\Z_2)\cong \HTI_q(\Sigma(C\vee D);\Z_2).
\end{equation}

\

At last, we present a recent result obtained in \cite[Theorem 5.19]{AKN} concerning the space of formal barycenters on a disjoint union of spaces.

\begin{proposition}\label{prop:AKN}
For $C$, $D$ two disjoint connected spaces and $k\geq2$, $Bar_k(C\amalg D)$ has the homology of

\[
\begin{split}
&Bar_k(C)\vee \Sigma Bar_{k-1}(C)\vee Bar_k(D) \vee\Sigma Bar_{k-1}(D)\vee\\
&\vee \bigvee_{\ell=1}^{k-1}(Bar_{k-\ell}(C)\ast Bar_{\ell}(D))\vee \bigvee_{\ell=2}^{k-1}(\Sigma Bar_{k-\ell}(C))\ast Bar_{\ell-1}(D).
\end{split}
\]
\end{proposition}

\section{Compactness of solutions}\label{sectcompac}
\setcounter{equation}{0}
In this section we present the proof of Theorem~\ref{thmcompa}, a compactness result for solutions $u_n$ of the general problem \eqref{equationbis}.

As commented previously, most of the results in this direction consider only the case of positive $K$, like for instance \cite{breme,lisha} for the regular case and \cite{barmon, bt} for the singular one. In order to prove Theorem \ref{thmcompa}, and following \cite{ChenLiAnn}, we will first derive an a priori estimate in the region in which $K$ is negative, and later we will give such estimates in the nodal region of $K$.

\begin{proposition} \label{propcompa1} Given $\delta>0$, there exists $C>0$ such that $ u_n(x) \leq C$ for all $x \in \Sigma^- \setminus \Gamma^\delta$, $n \in \N$. \end{proposition}

\begin{proposition} \label{propcompa2} There exist $\varepsilon, C>0$, such that $ u_n(x) \leq C$ for all $n \in \N$ and $x \in \Gamma^\varepsilon$. \end{proposition}

The proof of Theorem \ref{thmcompa} will be finally accomplished by studying the possible blow-up of the sequence $u_n$ in $\Sigma^+\setminus \Gamma^{\varepsilon}$.

\bigskip One of the difficulties in our study is that we do not know a priori whether the term

\beq \label{absvalue} \int_{\Sigma} |\tilde{K}_n| e^{u_n}dV_g \end{equation}
is bounded or not. By standard regularity results, this would give a priori $W^{1,p}$ estimates ($p \in (1,2)$) on $u_n$. Instead, if we integrate \eqref{equationbis} we only obtain that $\int_{\Sigma} \tilde{K}_n e^{u_n}$ is bounded.

Our first lemma shows that such kind of estimate is indeed possible for $u_n^{-}= \min\{ u_n,0\}$. This fact will be useful to prove both Propositions \ref{propcompa1} and \ref{propcompa2}.

\begin{lemma}\label{kato} Under the conditions of Theorem \ref{thmcompa}, define $v_n=u_n^{-} - \fint_{\Sigma} u_n^-$. Then there exists $C>0$ such that

\begin{enumerate}
\item[a)] $\|v_n\|_{L^p} \leq C$ for any $p \in [1,+\infty)$;
\item[b)] $v_n(x) \geq -C$ for any $x \in \Sigma$.
\end{enumerate}

\end{lemma}

\begin{proof}

We apply the well-known Kato inequality to the operator $\Delta_g$ (see for instance \cite[Theorem 5.1]{Ancona})

\beq \label{katito}
- \Delta_g u_n^{-} \geq  (-\Delta_g u_n)  \chi_{\{u_n \leq 0\}}  =\left (\tilde K_n e^{u_n} - f_n  \right) \chi_{\{u_n \leq 0\} } \geq -C h(x),
\end{equation}
where 

\beq \label{acha}
h(x) = 1 + \sum_{\substack{j \geq \ell+1 \\ \alpha_j <0} } d(x, p_j)^{2 \alpha_j}.
\eeq
Observe that $h \in L^q(\Sigma)$ for $q \in [1, 1+\delta)$ if $\delta >0 $ is sufficiently small.

Since the Radon measures $\mu_n = - \Delta_g u_n^{-} \geq -C h(x)$ are given as a divergence (in the sense of distributions), then $\int_{\Sigma} \, d\, \mu_n = 0$. From that we conclude that
$\int_{\Sigma} \, d\, |\mu_n |$ is bounded. We use the Green's representation for $v_n$ and H\"{o}lder inequality to obtain:

$$ |v_n(x)|^p \leq   \left ( \int_{\Sigma} \, d\, |\mu_n | \right )^{p-1} \int_{\Sigma} |G(x,y)|^p  d \, |\mu_n(y)| .$$

We now integrate in $x$ and make use of Fubini Theorem, taking into account that the Green function of $\Delta_g$ in $\Sigma$ belongs to $L^p$:

\begin{align*} \|v_n\|_{L^p}^p \leq  \left ( \int_{\Sigma} \, d\, |\mu_n | \right )^{p-1}
  \int_{\Sigma} \int_{\Sigma} |G(x,y)|^p  d \, |\mu_n(y)|   \, dx \\ =  \left ( \int_{\Sigma} \, d\, |\mu_n | \right )^{p-1}
  \int_{\Sigma} \int_{\Sigma} |G(x,y)|^p \, dx \ d \, |\mu_n(y)|  \leq C  \left ( \int_{\Sigma} \, d\, |\mu_n | \right )^{p}.\end{align*}

This concludes the proof of a). 

For the proof of b), we write the Green function of $\Delta_g$ in $\Sigma$ as $G(x,y)= -\frac{1}{2\pi} \log ( r d(x,y)) + \tilde H(x,y)$, where $\tilde H: \Sigma \times \Sigma \to \R$ is a bounded function. Here we have chosen $r \in (0, diam(\Sigma)^{-1})$. Then,

\begin{align*}
v_n(x)= \int_{\Sigma} G(x,y) d \, \mu_n(y) = -\frac{1}{2\pi} \int_{\Sigma} \log ( r d(x,y)) d \, \mu_n^+(y) \\ -\frac{1}{2\pi} \int_{\Sigma} \log ( r d(x,y)) d \, \mu_n^-(y) + \int_{\Sigma} \tilde H(x,y) d \, \mu_n(y).
\end{align*}

By the choice of $r>0$,

$$ -\frac{1}{2\pi} \int_{\Sigma} \log ( r d(x,y)) d \, \mu_n^+(y) \geq 0.$$

Moreover, by \eqref{katito},
$$ -\frac{1}{2\pi} \int_{\Sigma} \log ( r d(x,y)) d \, \mu_n^-(y) \geq - C \frac{1}{2\pi} \int_{\Sigma} \log ( r d(x,y)) h(y))\, d y \geq -C,
$$
and finally
$$ \left | \int_{\Sigma} \tilde H(x,y) d \, \mu_n(y) \right | \leq \|\tilde H\|_{L^{\infty}} \int_{\S} d |\mu_n| \leq C.$$
\end{proof}

As a first consequence of Lemma \ref{kato}, we present an integral estimate in domains entirely contained in the positive or negative region.
This result is an extension of the Chen-Li integral estimate for positive solutions, see \cite{ChenLiMPMS}. In our case $u_n$ may change sign, but we can perform the estimate thanks to Lemma \ref{kato}.

\begin{lemma}\label{intestimate}
Under the conditions of Theorem \ref{thmcompa}, for every open subdomain $\Sigma_0$ completely contained in $\Sigma^+$ or $\Sigma^-$, there exists $C>0$ so that

$$
\left | \displaystyle  \int_{\Sigma_0} \tilde{K}_n e^{u_n} dV_g\right | \leq C .
$$
\end{lemma}

\begin{proof}
Take $\Sigma_1$ a smooth domain such that $\overline{\Sigma_0} \subset \Sigma_1 \subset \overline{\Sigma_1} \subset \Sigma^{\pm}$.  Let $\varphi$ be the first eigenfunction of the Laplace operator in $\Sigma_1$, that is,

$$
\begin{cases}
 -\Delta_g \varphi = \lambda_1 \varphi  & \mbox{ in } \Sigma_1, \\ \varphi =0  & \mbox{ on } \partial \Sigma_1. \end{cases}
$$

Next, we multiply \eqref{equationbis} by $\varphi^2$, and integrate by parts over $\Sigma_1$ to obtain

\begin{equation}\label{1}
\displaystyle{ \int_{\Sigma_1} \tilde{K}_n \varphi^{2} e^{u_n} = -  \int_{\Sigma_1}  u_n \Delta_g (\varphi^2) } + O(1).
\end{equation}

Let us denote $f= \Delta_g (\varphi^2) = 2(|\nabla \varphi |^2  - \lambda_1 \varphi^{2})$. Observe that $\int_{\Sigma_1} f =0$. Then

$$ \int_{\Sigma_1}  u_n^{-} f =  \int_{\Sigma_1}  \left ( u_n^{-} - \fint_{\Sigma} u_n^- \right ) f,$$
so that, by Lemma \ref{kato}, a),

\begin{equation}\label{2} 
\left | \int_{\Sigma_1}  u_n^{-} f \right | \leq \left \|  u_n^{-} - \fint_{\Sigma} u_n^- \right \|_{L^1(\Sigma_1)} \|  f \|_{L^{\infty}(\Sigma_1)} \leq C.
\end{equation}
On the other hand, for any $\gamma>0$,

$$ \int_{\Sigma_1}  u_n^{+} |f| \leq C \int_{\Sigma_1}  u_n^{+} \leq  C  \int_{\Sigma_1}  u_n^{+} \frac{|\varphi^2 \tilde{K}_n|^{\gamma} }{|\varphi^2 \tilde{K}_n|^{\gamma}}.$$
By Young inequality we obtain

\begin{equation}\label{3}
 \int_{\Sigma_1}  u_n^{+} |f| \leq \e \int_{\Sigma_1} |u_n^{+}|^{\frac{1}{\gamma}} \varphi^{2} |\tilde{K}_n| + C_\e \int_{\Sigma_1} \frac{1}{|\varphi^2 \tilde{K}_n|^{\frac{\gamma}{1-\gamma}}}.
\end{equation}
We can take $\gamma >0$ sufficiently small so that the second integral term in the right hand side is finite (recall that, by Hopf lemma, $\varphi \sim d(x, \partial \Sigma_1)$ near the boundary). Then, by \eqref{1}, \eqref{2} and \eqref{3} 

$$
 \int_{\Sigma_1} |\tilde{K}_n| \varphi^{2} e^{u_n} \leq C +  \e \int_{\Sigma_1} |u_n^{+}|^{\frac{1}{\gamma}} \varphi^{2} |\tilde{K}_n|.
$$
We now use the inequality $(t^+)^{\frac{1}{\gamma}} \leq C + e^t$ to conclude that

$$
 \int_{\Sigma_1} |\tilde{K}_n| \varphi^{2} e^{u_n} \leq C,
$$
finishing the proof.

\end{proof}

\medskip In order to prove Proposition \ref{propcompa1}, we will need the following result, which is based on a mean value inequality for subharmonic functions.

\begin{lemma}\label{gtest}
Let $w$ be a function defined in $\Sigma_0 \subset \Sigma$, $x_0 \in \Sigma_0$, and assume that $-\Delta_g w(x) \leq A$ for all $x\ \in \Sigma_0$,  for some positive value $A>0$. Take $R>0$ such that
$$ R < \min \left \{ \frac 1 5 d(x_0, \partial \Sigma_0),\ \frac 1 2 diam (\Sigma_0) \right \}.$$ Then there exists $C>0$ depending only on $\Sigma_0$ and $A$ such that

$$ \sup_{x \in B_{x_0}(R/4)} w(x) \leq C \left( 1+ \fint_{B_{x_0}(R)} w \right ).$$

\end{lemma}

\begin{proof}
Define $v$ as the solution of the problem

$$ \left\{\begin{array}{ll}
-\Delta_g v = -A,  \qquad & \text{in $\Sigma_0$,}\\
v = 0, \qquad  &\text{on $\partial \Sigma_0 $.}
\end{array}\right.
$$

Clearly $v$ is smooth and $w+v$ is a subharmonic function. We now apply the mean value inequality for subharmonic functions (see \cite[Theorem~2.1]{li-schoen} for its version on manifolds) to conclude.

\end{proof}

\begin{proof}[Proof of Proposition~\ref{propcompa1}]

Take $\S_0 \subset \overline{\S_0} \subset \Sigma^-$, $x \in \S_0$ and fix $r>0$ sufficiently small. We apply Lemma ~\ref{gtest} to $w=u^+$ and we obtain
\begin{eqnarray*}
& \displaystyle{ \sup_{B_x(r) } u_n^+(x) \leq C+ C \int_{B_{x}(4r)} u_n^+   =C+  C \int_{B_{x}(4r)} \frac{u_n^+}{p} \frac{(-\tilde{K}_n)^{1/p}(x)}{(-\tilde{K}_n)^{1/p}(x)}} \\
& \displaystyle{ \leq C+  C \int_{B_{x}(4r)} e^{\frac{u_n}{p}} \frac{(-\tilde{K}_n)^{1/p}(x)}{(-\tilde{K}_n)^{1/p}(x)}  \leq
 C+  C \left( \int_{B_{x}(4r)} -\tilde{K}_n(x) e^{u_n} \right)^{1/p}  \left( \int_{B_{x}(4r)} \frac{1}{(-\tilde{K}_n)^{\frac{1}{p-1}}(x)} \right)^{\frac{p-1}{p}}} .
\end{eqnarray*}
It suffices to choose a large enough $p$ and use Lemma~\ref{intestimate} to conclude that $\displaystyle{ \sup_{B_x(R) } u_n^+(x)} < C$.


\end{proof}

We now turn our attention to Proposition \ref{propcompa2}. The proof follows the argument of \cite{ChenLiMPMS}, with the main difference that our solutions $u_n$ are not positive. This difficulty can be bypassed thanks to the following lemma, whose proof is based on Lemma \ref{kato}.

\begin{lemma} Under the hypotheses of Theorem \ref{thmcompa}, and given $\delta>0$, there exists $C>0$ such that

\beq\label{oscillation}
u_n(x_0)-u_n(x_1)\leq C
\eeq
for every $n \in \N$, $x_0 \in \Sigma^- \setminus \Gamma^{\delta}$, $x_1 \in \Sigma$.
Moreover, for any $r_0>0$, there exists $C>0$ such that

\beq\label{derivk}
|\nabla u_n(x) | \leq C \qquad \forall\, x \in \Sigma^- \setminus ( \Gamma^{\delta} \cup \bigcup_{i=\ell+1}^m B_{p_i}(r_0) ).
\eeq
\end{lemma}

\begin{proof} By Lemma \ref{kato}, b), we have that

\beq \label{x1} u_n(x_1) - \fint_{\Sigma} u_n^- \geq u_n^-(x_1) - \fint_{\Sigma} u_n^- \geq C. \eeq

Taking into account Lemma \ref{gtest}, we have that for any fixed $r\in(0,\frac{\delta}{2})$,

\begin{align*}  u_n(x_0)- \fint_{\Sigma} u_n^-
\leq C \left( 1+ \fint_{B_{x_0}(r)} \left( u_n(x)- \fint_{\Sigma} u_n^-    \right)  \right) \end{align*}

Moreover, by Proposition \ref{propcompa1}, $u_n(x) \leq u_n^-(x) + C$ for all $x \in B_{x_0}(r)$. Making use of Lemma \ref{kato}, a), we conclude

\beq \label{x0}
u_n(x_0)- \fint_{\Sigma} u_n^-  \leq  C \left( 1+ \fint_{B_{x_0}(r)} \left | u_n^-(x)- \fint_{\Sigma} u_n^-    \right | \right) \leq C
\eeq

This estimate, together with \eqref{x1}, shows that \eqref{oscillation} holds true.

We now turn our attention to the proof of \eqref{derivk}. Given $r_0>0$, take any $p>2$ and fix $x$ such that $B_x(r) \subset \Sigma^- \setminus ( \Gamma^{\delta} \cup \bigcup_{i=\ell+1}^m B_{p_i}(r_0) )$. Recall the inequality (see \cite[Theorem~9.11]{gt})

$$
\left \| u_n-\fint_{\Sigma} u_n^-\right \|_{W^{2,p}(B_{x}(\frac{r}{2}))}\leq C \left (||\tilde{K}_ne^{u_n}-f_n||_{L^{p}(B_x(r))}+\left \| u_n -\fint_{\Sigma} u_n^-\right \|_{L^{p}(B_x(r))} \right).
$$

Combining \eqref{x0} and Lemma~\ref{kato}, b), $u_n -\fint_{\Sigma} u_n^-$ is uniformly bounded in $L^{\infty}(B_x(r))$, whereas Proposition~\ref{propcompa1} implies that $\tilde{K}_ne^{u_n}-f_n$ is uniformly bounded on $B_x(r)$. Therefore, 

$$\left \|u_n-\fint_{\Sigma} u_n^-\right \|_{W^{2,p}(B_{x}(\frac r2))} \leq C.$$ 

In particular \eqref{derivk} holds.

\end{proof}

\begin{proof}[Proof of Proposition \ref{propcompa2}]

Since the proof is of local nature, we first pass to a problem in a planar domain. Given a point $p\in\Gamma$, we take a small neighborhood $U$ of p and an isothermal coordinate system $y =(y_1, y_2)$ centered at $p$ such that the metric $g$ takes the form $g = e^{\varphi}(dy_1^2 + dy_2^2)$ in $\Omega' \subset \mathbb{R}^2$, where $\varphi$ is smooth and $\varphi(0) = 0$. Consequently, for a subdomain $\Omega \subset \Omega'$, $u_n$ satisfies

$$
-\Delta u_n = e^{\varphi(y)}\tilde{K}_n(y) e^{u_n}-e^{\varphi(y)}f_n(y), \qquad \mbox{in $\Omega$}
$$
where $\Delta$ is the usual laplacian.
 
Let us define $u_{0,n}$ as the unique solution for the problem

\beq\label{ecuaux}
\left\{\begin{array}{ll}
\Delta u_{0,n} = e^{\varphi} f_n,  \qquad & \text{in $\Omega$,}\\
u_{0,n} = 0, \qquad  &\text{on $\partial \Omega$.}
\end{array}\right.
\eeq

If we now write the equation in the new variable $u_n - u_{0,n}$, which will be denoted again by $u_n$, we obtain:

\beq\label{eq2}
-\Delta u_n = W_n(y) e^{u_n} \qquad \mbox{ in $\Omega$, }
\end{equation}
where $W_n(y)=e^{\varphi(y)}e^{u_{0,n}(y)}\tilde{K}_n(y)$.

Moreover, $W_n\to W$ in $C^{2,\alpha}(\overline{\Omega})$. Assumption \ref{H1} is translated to $W$ in the form

$$ W \mbox{ is a }C^{2,\alpha}(\Omega) \mbox{ function, changes sign and  }\nabla W (x) \neq 0 \mbox{ in } \Gamma= \{ x \in \Omega: \ W(x)=0 \}.$$

Our proof is based on the method of moving planes, which allows us to compare the values of $u_n$ close to $\Gamma$. For the sake of clarity, we drop the subindex $n$ in the notation of the rest of this proof.

By the assumptions on $W$ for small $\delta>0$, there exists $\beta>0$ s.t.

\beq\label{ver1}
\left \vert \nabla W(y) \right\vert \geq \beta \, \mbox{ for any } y \mbox{ with } \vert W(y) \vert \leq \delta.
\eeq

For a given point $x_0\in\Gamma$, take a ball $B\subset \{y\in\Omega: W(y)>0\}$ tangent to $\Gamma$ at $x_0$. Applying the Kelvin transform which leaves $\partial B$ invariant and taking a neighborhood of $x_0$, we obtain a set (renamed as $\Gamma$) which is strictly convex (recall that $\Gamma$ is a $C^{2,\alpha}$ curve).

Next, through a translation and a rotation we define the new system as $x=(x_1,x_2)$ and $x_1=\gamma(x_2)$ which corresponds to the curve $\Gamma$. Let $\Omega_\e = \{x_1 < \gamma(x_2)+\e \} \cap \{x_1 > -2 \e \}$ and $\partial_l \OE=\left\{x:\ x_1-\gamma(x_2)=\varepsilon \right\}$. In summary, it is possible to choose a conformal map such that, for some $\varepsilon>0$ small, the following hold (see figure):
\begin{itemize}
\item[(i)] $x_0$ becomes the origin;
\item[(ii)] $\OE$ is located to the left of the line $x_1=\varepsilon$ and it is tangent to it;
\item[(iii)] $\partial_l \OE$ is uniformly convex;
\item[(iv)] $\frac{\partial W}{\partial x_1} \leq -\frac 12 \beta $, for every $x \in \OE$;
\item[(v)] $\overline{\OE} \cap \{p_{1},\ldots,p_m\} = \emptyset$.
\end{itemize}

 
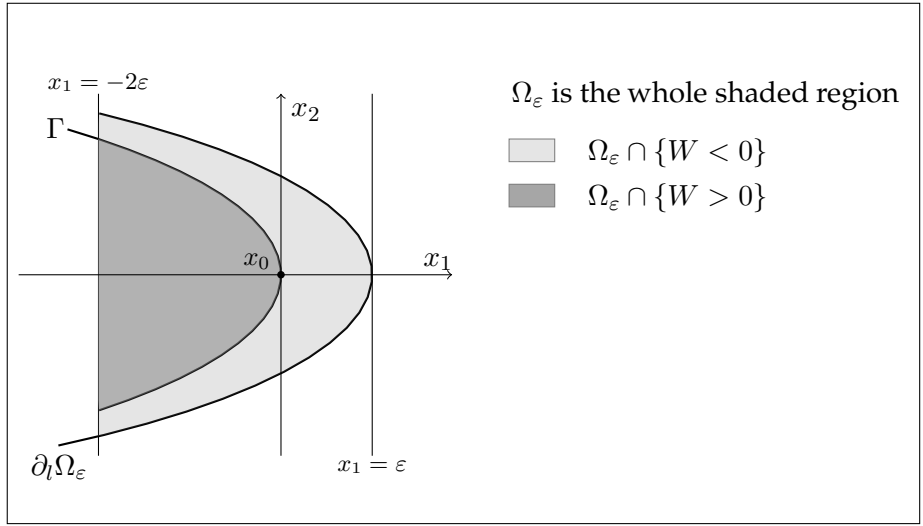
\bigskip

\begin{center}
\begin{tikzpicture}[scale=0.6]

\draw[black,thick,domain=140:216.8] plot ({-20-20*cos(\x)}, {5*sin(\x)}) (-4.5,3.25) node[left] {$\Gamma$};
\draw[black,thick,domain=134.5:229] plot ({-18-20*cos(\x)}, {5*sin(\x)}) (-4,-4.25) node[left] {$\partial_{l}\Omega_{\varepsilon}$};

\path[draw] (-4,-4) -- (-4,4) (-4,4.25) node {\footnotesize $x_1=-2\varepsilon$};

\path[draw] (2,-4) -- (2,4) (2,-4.25) node {\footnotesize $x_1=\varepsilon$};

\begin{scope}
\clip[domain=140:216.8] plot ({-20-20*cos(\x)}, {5*sin(\x)});
\fill[gray, fill opacity=0.5] (-4,-10) rectangle (19,6);
\end{scope}

\begin{scope}
\clip[domain=134.5:229] plot ({-18-20*cos(\x)}, {5*sin(\x)}) ;
\fill[gray, fill opacity=0.2] (-4,-10) rectangle (10,5);
\end{scope}

\draw[->,color=black] (-5.75,0) -- (3.75,0);
\draw[->,color=black] (0,-4) -- (0,4);
\draw (2.9,0.7) node[anchor=north west] {$x_1$};
\draw (0,4) node[anchor=north west] {$x_2$};

\draw [fill=black] (0,0) circle (2pt) (0,0.3) node[left] {\small $x_0$};

\draw  (-6,-5.5) rectangle (14,6) (9.3,4.5) node[below] {$\Omega_{\varepsilon}$ is the whole shaded region} (6.5,2.7) node[right,color=black] {$\Omega_{\varepsilon} \cap \{W<0\}$} (6.5,1.7) node[right,color=black] {$\Omega_{\varepsilon} \cap \{W>0\}$};
 
\filldraw[draw=black,color=gray,fill opacity=0.2] (5.,2.5) rectangle  (6,3);
\filldraw[draw=black,color=gray,fill opacity=0.7] (5.,1.5) rectangle  (6,2);

\end{tikzpicture}

\end{center}

\bigskip


Let $m=\displaystyle{\min_{x \in \partial_l \OE} u(x)}$ and $M=\displaystyle{\max_{x \in \partial_l \OE} u(x)}$. We define $\tilde{u} \in C^2(\overline{\OE})$ such that:
\begin{enumerate}
\item[a)] $\tilde{u}=u$ in $\partial_l \OE$;
\item[b)] $m \leq \tilde{u} \leq M$;
\item[c)] $|\nabla \tilde{u} | \leq C$ in $\OE$. 
\end{enumerate}
Observe that c) is possible by \eqref{derivk}.

Let $w$ be the harmonic function

\beq\label{ecuaux}
\left\{\begin{array}{ll}
\Delta w = 0,  \qquad & \text{in $\OE$,}\\
w = \tilde{u}, \qquad  &\text{on $\partial\OE $.} 
\end{array}\right.
\eeq

Due to \eqref{oscillation}, the oscillation of $u$ on $\partial_l\OE$ is bounded, i.e.,

\beq\label{oscil}
M-m=\displaystyle{\max_{\partial_l\OE} u - \min_{\partial_l\OE} u \leq C}.
\eeq
Consequently, the oscillation of $w$ is also bounded in $\OE$. We also define a new auxiliary function $v$ as

\beq\label{funv}
v(x)=u(x)-w(x)+C_0(\eps+\gamma(x_2)-x_1),
\eeq
for some $C_0>0$ to be determined. It is clear that the function $v$ verifies

\beq\label{ecuaux2}
\Delta v + f(x,v(x))-C_0 \gamma''(x_2)= 0,  \qquad \text{in $\O_{\eps}$,}
\eeq
with
$$
f(x,v(x))=W(x)e^{v(x)+w(x)-C_0(\eps+\gamma(x_2)-x_1)}.
$$

We claim that for a suitable $C_0$

\beq\label{known}
v(x)\geq 0 \mbox{ in } \O_{\eps} \quad \mbox{ and } \quad v(x)=0 \mbox{ on }\partial_l\OE.
\eeq
The boundary condition is direct. In order to prove the first part, we distinguish two cases:

\begin{itemize}
\item{\textbf{Case 1}: $\frac{\eps}{2}< x_1-\gamma(x_2)\leq\eps$}

\

Taking into account (v), by \eqref{derivk} we have that
$$
\displaystyle{  \left\vert \frac{\partial u}{\partial x_1} \right \vert \leq C} \, \mbox{ and } \, \displaystyle{ \left\vert \frac{\partial w}{\partial x_1}\right\vert \leq C}.
$$
Consequently,

\beq\label{ineq1}
\frac{\partial v}{\partial x_1}=\frac{\partial u}{\partial x_1} - \frac{\partial w}{\partial x_1}-C_0 \leq \ C-C_0.
\eeq
It suffices to choose $C_0$ sufficiently large to obtain that $\frac{\partial v}{\partial x_1}$ is negative. Since $v=0$ on $\partial_l\OE$, it is clear that \eqref{known} holds.

\

\item{\textbf{Case 2}: $x_1-\gamma(x_2)\leq \frac{\eps}{2}$ and $x_1 \geq -2\eps$}

\

By \eqref{oscillation}, we have that
$$
\displaystyle{ v(x)=u(x)-w(x)+C_0(\eps+\gamma(x_2)-x_1) \geq \min_{\OE} u - \max_{\partial_l \OE } u + C_0 \frac{\eps}{2} \geq -C + C_0 \frac{\eps}{2}}.
$$

So, choosing $C_0$ sufficiently large, \eqref{known} holds.

\end{itemize}

\

Now we are ready to apply the method of moving planes to $v$ in the $x_1$ direction. Thus, we start from $x_1=\eps$ and move the line perpendicular to $x_1-$axis towards the left. Namely, let $T_{\lambda}=\left\{ x \in \mathbb{R}^2: \ x_1 \geq \lambda \right\}$ the half-plane, $M_{\lambda}=\left\{ x \in \mathbb{R}^2: \ x_1 = \lambda \right\}$ its boundary and $x_{\lambda}=(2\lambda-x_1,x_2)$ the reflection point of $x$ with respect to the line $M_{\lambda}$. Our goal is to prove that

\beq\label{gpr}
v(x_{\lambda}) \geq v(x),
\eeq
for every $x\in T_{\lambda} \cap \O_{\eps}$ for $\lambda \in \left[\frac{\eps}{2}-\eps_1,\eps\right]$, with some $\eps_1\in\left(0,\eps\right)$ to be determined. By \eqref{known} and \eqref{ineq1}, \eqref{gpr} holds for $\lambda\in \left(\frac{3}{4}\eps,\eps\right]$.

By a standard argument (see \cite{GNN}), we see that the moving planes argument can be carried on provided that
\beq\label{monf}
f(x,v)\leq f(x_{\lambda},v) \quad \mbox{ for every } x \in \O_{\eps} \quad \mbox{ with } \quad \eps>\lambda>\frac{\eps}{2}-\eps_1.
\eeq

It is easy to check that \eqref{monf} holds if

\beq\label{monf2}
\frac{\partial f(x,v(x))}{\partial x_1}=e^u\left(\frac{\partial W}{\partial x_1} +W\left(\frac{\partial w}{\partial x_1} + C_0 \right) \right) \leq 0, \, \quad \mbox{ if } x \in {\Omega}_\eps, \ x_1 > -\eps_1.
\eeq
In other words, if $f$ is monotone decreasing along the direction $(1,0)$, near $x_0$.

If $W(x) \leq 0$, it is enough to choose $C_0 >- \frac{\partial w }{\partial x_1} $ to verify
$$
\frac{\partial W}{\partial x_1} +W\left(\frac{\partial w}{\partial x_1} + C_0 \right) \leq W \left(\frac{\partial w}{\partial x_1} + C_0  \right) \leq 0.
$$

In the case that $W(x)>0$ by the assumptions on $W$, for every $\eps_1>0$ there exists a neighborhood $V_{\eps_1}$ of $\Gamma$ such that $W(x)\leq \eps_1 $.

Since $\frac{\partial w}{\partial x_1} + C_0 $ is bounded from above, then
$$
\frac{\partial W}{\partial x_1} + W\left(\frac{\partial w}{\partial x_1} + C_0 \right) \leq -\frac{\beta}{2} +\eps_1 \left(\frac{\partial w}{\partial x_1} + C_0  \right)\leq -\frac{\beta}{2}+ \eps_1 C,
$$
therefore we can take $\eps_1$ small enough to obtain the desired conclusion. We choose $\eps_1<\eps$.

In this way, the method of moving planes works up to $\lambda=\frac{\eps}{2} -\eps_1$. Therefore, \eqref{gpr} implies that $v(x)$ is monotone decreasing in the $(1,0)-$direction. In fact, we can repeat the previous argument rotating the $x_1-$axis by a small angle.

Thus, there exist $\eps_2>0$ and a fixed cone $\Delta_0$ such that for any $x\in B_{x_0}(\eps_2)$ we have

$$
v(y) \geq  v(x), \, \forall\; y \in \Delta_{x},
$$
and

\begin{equation}\label{nonempty}
\emptyset\neq\Delta_x\cap\Sigma^+\subset\Omega_\varepsilon
\end{equation}
where $\Delta_x$ denotes a translation of the cone $\Delta_0$ with $x$ at its vertex. By \eqref{funv}, we can transform the previous inequality into

\beq\label{monf3}
u(y) + C(\eps_1) \geq u(x) \quad \forall \,y \in \Delta_{x}.
\eeq
Moreover, there exists $\eta>0$, such that for any $x\in B_{x_0}(\eps_2)$ the intersection of the cone $\Delta_x$ with the set $\Sigma^+\setminus\Gamma^\eta$ has a positive measure, and the lower bound of the measure depends only on $\eta$ and the $C^1$ norm of $K$. Namely, setting $\Sigma_x=\Delta_x\cap(\Sigma^+\setminus\Gamma^\eta)$ we have that for any $x\in B_{x_0}( \eps_2)$

\begin{equation}\label{l1}
|\Sigma_x|\geq\eta_1>0.
\end{equation}
Thanks to this, the proof can be concluded combining \eqref{monf3} and the integral bound of Lemma \ref{intestimate}. Indeed by virtue of \eqref{nonempty} and property (v)

\begin{equation}\label{l2}
\min_{\Sigma_x}\tilde K\geq \eta_2>0,
\end{equation}
hence

\[
e^{u(x)}=\fint_{\Sigma_x}e^{u(x)}dy\overset{\eqref{monf3}}{\leq} C\fint_{\Sigma_x} e^{u(y)}dy\overset{\eqref{l2}}{\leq} C\fint_{\Sigma_x}  \tilde K e^{u(y)}dy \overset{\textnormal{Lemma \ref{intestimate}}}{\leq} \frac{C}{|\Sigma_x|} \overset{\eqref{l1}}{\leq} C.
\]

\end{proof}

\medskip

\begin{proof}[Proof of Theorem~\ref{thmcompa}]
Take $\varepsilon>\delta>0$ and the open set $\Sigma_1=\S^+ \setminus \overline{\Gamma^{\delta}}$, where $\varepsilon$ is given by Proposition~\ref{propcompa2}. By Propositions \ref{propcompa1} and \ref{propcompa2}, $u_n$ is uniformly bounded from above in $\Sigma \setminus \Sigma_1$. Moreover, by Lemma \ref{intestimate}, $\int_{\Sigma_1} \tilde{K}_n e^{u_n}$ is bounded. By \cite{bt} there are two possibilities:

\bigskip

{\bf Case 1}: $u_n$ is bounded from above in $\Sigma$. Therefore, $\tilde K_n e^{u_n} -f_n \in L^{p}(\Sigma \setminus \Sigma_1)$ for some $p>1$.  Elliptic regularity estimates imply that $u_n-\fint_{\Sigma} u_n \in W^{2,p}(\Sigma)$, so $u_n-\fint_{\Sigma} u_n \in L^{\infty}(\Sigma)$. If $\fint_{\Sigma} u_n$ is bounded, we obtain (1); if, on the contrary, a subsequence of $\fint_{\Sigma} u_n$ diverges negatively, we obtain (2).

\bigskip

{\bf Case 2}: The sequence $u_n$ is not bounded from above. Applying the results of \cite{bt} concerning the blow--up analysis for \eqref{equationbis} in $\Sigma_1$, we can assume that there exists a finite blow-up set $S=\{q_1,\ldots,q_r\} \subset \Sigma_1$. Moreover, by enlarging $\delta$ if necessary, we can assume that $u_n \to -\infty$ uniformly in $\partial \Sigma_1$, and
$$
\tilde{K}_n e^{u_n} \rightharpoonup \sum_{i=1}^r \beta(q_i) \delta_{q_i} \quad \mbox{ in the sense of weak convergence of measures in }
 \overline{\Sigma_1}, $$
with $\beta(q_i) \geq 8\pi$.

Now, let us define $v$ the solution of the problem

$$
\begin{cases}
 \begin{array}{ll} -\Delta_g v = C_1 h   & \mbox{ in } \Sigma \setminus \Sigma_1, \\ v =0  & \mbox{ on } \partial \Sigma_1, \end{array}
\end{cases}
$$
where $h$ is defined in \eqref{acha} and $C_1$ is a positive constant such that $C_1 h$ is an uniform upper bound of the term $\tilde K_n e^{u_n} -f_n$ in $\Sigma \setminus \Sigma_1$. Since $ h \in L^{p}(\Sigma \setminus \Sigma_1)$ for some $p>1$, then $ v \in L^{\infty}(\Sigma \setminus \Sigma_1)$ by standard regularity results. By the maximum principle, for any $C>0$ there exists $n_0 \in \N$ such that $u_n \leq v-C$ in $\Sigma \setminus \Sigma_1$ for $n\geq n_0$. This implies that $u_n \to -\infty$ uniformly in $\Sigma \setminus \Sigma_1$; in particular,

\beq\label{mconv}
\tilde{K}_n e^{u_n} \rightharpoonup \sum_{i=1}^r \beta(q_i) \delta_{q_i} \quad \mbox{ in the sense of weak convergence of measures in $\Sigma$}.
\end{equation}

It is worth to point out that, at this point of the proof, we can not apply yet the quantization part of the concentration-compactness Theorem of \cite{bt} unless we check the bounded oscillation condition on $\partial \Sigma_1$.

By \eqref{mconv}, employing the Green's representation formula for $u_n$, we have that

$$
u_n-\overline{u_n} \to \sum_{i=1}^r \beta(q_j) G(x,q_j) + h_m,
$$
uniformly on compact sets of $\Sigma\setminus (S\cup\{p_1,\ldots,p_{\ell} \} )$, where $h_m$ is defined in \eqref{accame}. Therefore, the sequence $u_n-\overline{u_n}$ admits uniformly bounded oscillation on any compact set of $\Sigma\setminus (S\cup\{p_1,\ldots,p_{\ell} \} )$. Indeed, there exists a constant $C>0$ such that

$$
\max_{\partial \Sigma_1} u_n-\min_{\partial \Sigma_1} u_n <C.
$$

By virtue of this condition, we can apply the quantization result of \cite{bt} to conclude that, up to subsequence,

$$
\lim_{k\to +\infty} \int_{\Sigma} \tilde{K}_n e^{u_n} \in \Lambda.
$$
\end{proof}

%

\section{Topological description of the energy sublevels and Morse inequalities}\label{sectproj}

\setcounter{equation}{0}

In this section we study the topology of the energy sublevels of $I_\lambda$. We shall observe a change in the topology of the sublevels between high and low ones. This fact will be decisive to prove the existence and multiplicity theorems.

First, we define a continuous projection $\Psi$ from low sublevels of $I_{\lambda}$ onto a compact topological space, whose character is inherited from the geometry of the function $K$. By using ideas from \cite{MalchiodiRuizSphere} we give a more accurate description for $k=1$ depending on the order of the conical points.

Next, we define the reverse map $\varphi$ from the corresponding space onto $I_{\lambda}^{-L}=\{u\in \overline{X} \,:\,I_\lambda(u) \leq -L\}$ with $L>0$ large enough such that the composition $\Psi\circ\varphi$ is homotopically equivalent to the identity map.

Finally, we adapt the well-known Morse inequalities for $I_{\lambda}$ which will be crucial to prove the multiplicity theorems.

Since this scheme has been used several times, see \cite{DeMLS,DeM1,DjadliMalchiodi,MalchiodiRuizSphere}, we will be sketchy and focus on the main differences concerning the sign-changing case.

\subsection{Topological description of the low sublevels of $I_\lambda$}

\

The first result allows us to project continuously functions $u$ with a low energy level onto the set of formal barycenters on a union of bouquets and a simplex contained in $\Sigma^+$.

\begin{proposition}\label{prop:lowsublevelsgeneral}
Let $\lambda\in(8k\pi, 8\pi(k+1))$, $k\in\N$, and assume \ref{H1}, \ref{H2}. Then for $L>0$ sufficiently large there exists a continuous projection

\[
\Psi: I_\lambda^{-L}\longrightarrow Bar_k(Z_{N,M}),
\]
where

\beq\label{ZNM}
Z_{N,M}=\coprod_{i=1}^N B^{g_i} \amalg Y_M \subset \Sigma^+\setminus\{p_1,\cdots,p_{\ell}\} \quad \mbox{ and } \quad Y_M=\{y_1,\ldots,y_M\},
\end{equation}
where $B^{g_i}\subset A_i$ is a bouquet of $g_i$ circles, with $g_i$ defined in \eqref{concomp}, and $y_h\in C_h$.\\
Moreover, if $\frac{\tilde{K}^+e^{u_n}}{\int_{\Sigma}\tilde{K}^+e^{u_n}dV_g}\rightharpoonup\sigma$, for some $\sigma\in Bar_k(Z_{N,M})$, then $\Psi(u_n)\to\sigma$.
\end{proposition}

\begin{remark}
Under assumption \ref{H3}, the topological set $Bar_k(Z_{N,M})$ is not contractible. In case $N=0$, $Bar_k(Z_{N,M})$ is the $(k-1)$-skeleton of $(M-1)$-symplex, which is not contractible if $k<M$ (see Exercise 16 in Section~2.2 of \cite{Hatcher}).
\end{remark}

\begin{proof}[Proof of Proposition~\ref{prop:lowsublevelsgeneral}]

This lemma is proved in the spirit of \cite{DjadliMalchiodi}, but following closely the approach of \cite{BJMR}.

\textbf{Claim:} If $I_\l(u_n)\to-\infty$, up to a subsequence,
 
$$
\sigma_n:=\frac{\tilde{K}^+ e^{u_n}}{\int_{\Sigma} \tilde{K}^+ e^{u_n}dV_g}\rightharpoonup\sigma \in Bar_k(\overline{\S^+}).
$$

Suppose by contradiction that there exist $k+1$ points $x_1,\ldots,x_{k+1} \subset supp(\sigma)$. Take $r>0$ such that $B_{x_i}(2r)\cap B_{x_j}(2r)=\emptyset$ for $i\neq j$. Therefore, there exists $\varepsilon>0$ such that $\sigma(B_{x_i}(2r))>2\varepsilon$. As a consequence, $\sigma_n(B_{x_i}(r))\geq \varepsilon$; this implies that for any $\tilde{\eps}>0$:

\begin{equation} \label{nueva}
 \log \int_\Sigma \tilde{K}^+ e^{u_n} - \frac{1}{|\Sigma|} \int_\Sigma {u_n}    \leq \frac{1}{(16(k+1)\pi-\tilde{\eps} ) } \int_\Sigma |\nabla u_n|^2 \,dV_g + C,
\end{equation}
where $C$ is a positive constant which depends on $\eps,r,\tilde{\eps}$.  This kind of improvements of the Moser-Trudinger inequality were first obtained by Chen and Li in \cite{ChenLiIneq}. We also refer to \cite[Proposition 2.2 and Proposition 2.3]{MalchiodiRuizSphere}, where a formulation of this inequality for nonnegative functions was given in the case $k=2$ (the case $k>2$ is analogous). 

Taking $\tilde \varepsilon$ sufficiently small, \eqref{nueva} violates the hypothesis that $I_\l(u_n)$ diverges negatively, yielding a contradiction.

\bigskip

By the claim, given a neighborhood $V$ of $Bar_k(\S^+)$ in the weak topology of measures, there exists $L_0>O$ large enough such that if $L>L_0$, then
\beq\label{PP}
\frac{\tilde{K}^+ e^u}{\int_{\S} \tilde{K}^+ e^u \, dV_{g}}\in V, \quad \forall\: u \in I_{\l}^{-L}.
\eeq

In the appendix of \cite{BJMR}, it is proved that $Bar_k(\overline{\S^+})$ is a Euclidean Neighborhood Retract. Observe that the weak topology of measures is metrizable on bounded sets, see \cite[Theorem~3.28]{Bre}. By \cite[Lemma~E.1]{Bredon}, there exists $V$ a neighborhood of $Bar_k(\overline{\S^+})$ in the weak topology of measures, and a continuous retraction $\mathcal{X}:V\to Bar_k(\overline{\S^+})$. Next, by \eqref{PP}, we define $\tilde{\Psi}$ as

$$
\begin{array}{cccccc}
\tilde{\Psi}:&I^{-L}_\lambda&\longrightarrow &V&  \xrightarrow{\mathcal X}&Bar_k(\overline{\S^+})\\
 &u&\longmapsto&\frac{\tilde{K}^+e^u}{\int_\S \tilde{K}^+e^u \, dV_{g}}& \longmapsto&\sum_{i=1}^k t_i\delta_{x_i}.
\end{array}
$$

Observe that we can retract continuously $\overline{A_i}$ onto $B^{g_i}$ and $\overline{C_h}$ onto a single point $y_h\in C_h$. Consequently, we can define the retraction

\beq\label{r}
r: \overline{\Sigma^+ }\longrightarrow \displaystyle{Z_{N,M}}.
\end{equation}

We are now in conditions to define the map $\Psi$ as the composition of $\tilde{\Psi}$ with the function $r*: Bar_k(\overline{\Sigma^+}) \longrightarrow Bar_k(Z_{N,M})$, the pushforward induced by the function $r$, then

\[
\begin{array}{cccc}
\Psi:&I^{-L}_\lambda&\longrightarrow&Bar_k(Z_{N,M} )\\
 &u&\longmapsto&\sum_i s_i\delta_{x_i},
\end{array}
\]
where the values $s_i$ are given by $\tilde\Psi$. Since $r$ is a retraction, if $\frac{\tilde{K}^+ e^{u_n}}{\int_{\Sigma}\tilde{K}^+  e^{u_n}dV_g}\rightharpoonup\sigma$, for some $\sigma\in Bar_k(Z_{N,M})$, then $\Psi(u_n)\to\sigma$.

\end{proof}

On the other hand, for $\lambda\in(8k\pi, 8\pi(k+1))$, $k\in\N$, and $Z$ a compact subset of $\Sigma^+\setminus\{p_1,\ldots,p_\ell\}$ we consider test functions \emph{concentrated} in at most $k$ points of $Z$ with arbitrary low energy. For $\gamma>0$ small enough, we consider a smooth nondecreasing cut-off function $\chi_\gamma:\R^+\to\R^+$ such that

\beq\label{chib}
\chi_\gamma(t)=\left\{
\begin{array}{ll}
t & \mbox{for $t\in[0,\gamma]$}\\
2\gamma & \mbox{for $t\geq 2\gamma$.}
\end{array}
\right.
\end{equation}

For $\mu>0$ and $\sigma=\sum_{i=1}^k t_i\delta_{x_i}\in Bar_k(Z)$, we define

$$
\phi_{\mu,\sigma}:\Sigma\to\R\qquad \phi_{\mu,\sigma}(x)=\log\sum t_i\left(\frac{\mu}{1+(\mu\chi_\gamma(d(x,x_i)))^2}\right)^2,
$$
$$
\varphi_{\mu,\sigma}(x)=\phi_{\mu,\sigma}(x)-\int_\Sigma \phi_{\mu,\sigma}\,dV_g.
$$

By the results in \cite[Lemma 4.11, Lemma 4.12]{DeMLS}, which hold independently of the genus of the surface, we have that

\begin{lemma}\label{lemma:testfunctionsgeneral}
Let $\lambda\in(8k\pi,8(k+1)\pi)$, $k\in\N$, and let $Z$ be a compact subset of $\Sigma^+\setminus\{p_1,\ldots,p_\ell\}$. Then we can choose $\gamma>0$ such that:

\begin{itemize}
\item[\emph{(i)}] given $L>0$ there exists a large $\mu(L)>0$ satisfying that for $\mu\geq\mu(L)$, $\varphi_{\mu,\sigma}\in\bar X$, where $\bar X$ is defined in \eqref{Xbar}, and $I_\lambda(\varphi_{\mu,\sigma})<-L$ for any $\sigma\in Bar_k(Z)$;
\item[\emph{(ii)}] for any $\sigma\in Bar_k(Z)$

\[
\frac{\tilde{K}^+e^{\varphi_{\mu,\sigma}}}{\int_{\Sigma}\tilde{K}^+ e^{\varphi_{\mu,\sigma}}dV_g}\rightharpoonup\sigma\qquad\mbox{as $\mu\to+\infty$.}
\]
\end{itemize}
\end{lemma}

The following two results allow us to deal with the case in which $\Sigma^+$ has  only simply connected components and $N^+\leq k$; however, it is restricted to $\lambda \in (8\pi,16\pi)$. 

\begin{proposition}\label{prop:lowsublevels8pi16pi}
Let $\lambda\in(8\pi,16\pi)$ and assume \ref{H1}, \ref{H2}.
Let

\beq
\label{JlambdaAC}
J_{\lambda,A_i}= J_\lambda\cap A_i\quad \mbox{for $i=1,\ldots,N$,}\qquad J_{\lambda,C_h}=J_\lambda\cap C_h\quad\mbox{for $h=1,\ldots,M$},
\end{equation}
where $J_\lambda$ is defined in \eqref{Jlambda}, and let us assume that, up to reordering, $M_\lambda\in\{1,\ldots,M\}$ is such that $J_{\lambda,C_h}\neq\emptyset$ if $h\in\{1,\ldots,M_\lambda\}$ and $J_{\lambda,C_h}=\emptyset$ if $h\in\{M_\lambda+1,\ldots,M\}$.\\
Then for $L>0$ sufficiently large there exists a continuous projection

\[
\Psi:I_{\lambda}^{-L}\to Bar_1(W_{N,M,J_{\lambda}}),
\]
where

\beq\label{WNMJ}
W_{N,M,J_{\lambda}}=\coprod_{i=1}^N B^{g_i+|J_{\lambda,A_i}|} \amalg\coprod_{h=1}^{M_\lambda} B^{|J_{\lambda,C_h}|}\amalg \hat Y_{M_{\lambda}},
\end{equation}
$B^{g_i+|J_{\lambda,A_i}|}\subset A_i$, $B^{|J_{\lambda,C_h}|} \subset C_h$ are bouquets of $g_i+|J_{\lambda,A_i}|$ and $|J_{\lambda,C_h}|$ circles respectively, with $g_i$ defined in \eqref{concomp}, for $i=1,\ldots,N$ and $h=1,\ldots,M_\lambda$, and $\hat Y_{M_{\lambda}}=\coprod_{h=M_\lambda+1}^M\{y_h\}$ with $y_h\in Y_h$ for $h=M_{\lambda}+1,\ldots,M$. \\ Moreover, if

\[
\frac{\tilde{K}^+ e^{u_n}}{\int_{\Sigma}\tilde{K}^+ e^{u_n}dV_g}\rightharpoonup \sigma, \qquad\mbox{for some $\sigma\in Bar_1(W_{N,M,J_{\lambda}})$,}
\]
then $\Psi(u_n)\to \sigma$.
\end{proposition}

\begin{remark}
Observe that $W_{N,M,J_{\lambda}}$ is not contractible if and only if either $N\geq 1$ or $M_{\lambda}>1$, i.e. if \ref{H3} holds, or if $J_{\lambda}\neq\emptyset$, namely \ref{H4} holds.
\end{remark}

\begin{proof}
For $L>0$ sufficiently large and $r>0$, we apply some results of \cite{DeMLS}; more specifically, by Propositions~4.4.,~4.7.,~4.8., and Remark~4.10, of \cite{DeMLS}, we construct the continuous projection

$$
\displaystyle{\beta:I_{\lambda}^{-L}\to\overline{\Sigma^+}\setminus\bigcup_{p_i\in J_{\lambda}}B_{p_i}(r)},
$$
with the property that if $\frac{\tilde{K}^+ e^{u_n}}{\int_{\Sigma}\tilde{K}^+ e^{u_n}dV_g}\rightharpoonup\delta_x$ for some $x\in\displaystyle{\overline{\Sigma^+}\setminus\bigcup_{p_i\in J_{\lambda}}B_{p_i}(r)}$ then $\beta(u_n)\to x$. Notice that $I_{\lambda}$ is bounded from below on the functions belonging to $\tilde{\Psi}^{-1}(J_{\lambda})$.

We can rewrite $\displaystyle{\overline{\Sigma^+}\setminus\bigcup_{p_i\in J_{\lambda}}B_{p_i}(r)}$ as

$$
\coprod_{i=1}^N A'_i \amalg\coprod_{h=1}^{M_\lambda}C'_h\amalg\coprod_{h=M_\lambda+1}^M C_h.
$$
where $A'_i =A_i \setminus \bigcup_{p_i\in J_{\lambda,A_i}} B_{p_i}(r)$ and $C'_h= C_h \setminus \bigcup_{p_i\in J_{\lambda,C_h}} B_{p_i}(r)$.

The sets $A'_i$ can be retracted to an inner bouquet $B^{g_i+|J_{\lambda,A_i}|} \subset A'_i$ and $C'_h$ to $B^{|J_{\lambda,C_h}|} \subset C'_h$, in a similar way to the proof of Proposition~\ref{prop:lowsublevelsgeneral}, we can define a retraction

\[
r:\overline{\Sigma^+}\setminus\bigcup_{p_i\in J_{\lambda}}B_{p_i}(r) \longrightarrow W_{N,M,J_{\lambda}}.
\]

Finally, we can define $\Psi$ as the composition of $\beta$ with the pushforward $r*:Bar_1(\overline{\Sigma^+}\setminus\bigcup_{p_i\in J_{\lambda}}B_{p_i}(r))\longrightarrow Bar_1(W_{N,M,J_{\lambda}})$, then

\[
\begin{array}{cccc}
\Psi:&I^{-L}_\lambda&\longrightarrow&Bar_1(W_{N,M,J_{\lambda}})\\
 &u&\longmapsto& \delta_{x}.
\end{array}
\]

Since $r$ is a retraction, the second part of the proposition is proved.

\end{proof}

Next, for $\lambda\in(8\pi,16\pi)$, we introduce appropriate test functions that will allow to map a compact subset $W$ of $\Sigma^+\setminus J_\lambda$ into low sublevels of $I_\lambda$.\\

Let $\tilde\alpha=\max_{n\leq\ell\,|\,p_n\notin J_\lambda}\alpha_n$ or $\tilde\alpha=0$ if $J_\lambda=\{p_1,\ldots,p_\ell\}$ or $\ell=0$. For any $\alpha\in(\tilde\alpha,\frac{\lambda}{8\pi}-1)$, $\mu>0$ and $p\in W$, we define

\[
\phi_{\mu,p,\alpha}:\Sigma\to\R, \quad \phi_{\mu,p,\alpha}(x)=2\log\left(\frac{\mu^{1+\alpha}}{1+(\mu\chi_\gamma(d(x,p)))^{2(1+\alpha)}}\right),
\]
\[
\varphi_{\mu,p,\alpha}(x)=\phi_{\mu,p,\alpha}(x)-\int_\Sigma \phi_{\mu,p,\alpha}\,dV_g.
\]
where $\chi_\gamma$ is defined in \eqref{chib}.

Since $W$ is a compact subset of $\Sigma^+\setminus J_\lambda$ the results in \cite[Lemma 4.13, Lemma 4.14]{DeMLS} holds, namely
\begin{lemma}\label{lemma:testfunctions8pi16pi}
Let $\lambda\in(8\pi,16\pi)$ and let $W$ be a compact subset of $\Sigma^+\setminus J_\lambda$. Then we can choose $\gamma>0$ such that:
\begin{itemize}
\item[\emph{(i)}] given any $L>0$, there exists a large $\mu(L)>0$ satisfying that, for any $\mu\geq\mu(L)$, $\varphi_{\mu,p,\alpha}\in \bar X$ and $I_\lambda(\varphi_{\mu,p,\alpha})<-L$ for any $p\in W$;
\item[\emph{(ii)}] for any $p\in W$,

\[
\frac{\tilde{K}^+e^{\varphi_{\mu,p,\alpha}}}{\int_{\Sigma}\tilde{K}^+ e^{\varphi_{\mu,p,\alpha}}dV_g}\rightharpoonup\delta_p\qquad\mbox{as $\mu\to+\infty$}.
\]
\end{itemize}
\end{lemma}

\subsection{Topological characterization of the high sublevels of $I_\lambda$}

\

The compactness result Theorem \ref{thmcompa}, combined with the deformation Lemma \cite[Proposition 2.3]{Lucia}, allows us to prove the next alternative bypassing the Palais-Smale condition, which is not known for the functional $I_\lambda$.

\begin{lemma}\label{lemma:deformation}
Let $\lambda\notin \Lambda$ and assume \ref{H1}, \ref{H2}. If $I_\lambda$ has no critical levels inside some interval $[a,b]$, then $I_\lambda^{a}$ is a deformation retract of $I_\lambda^b$.
\end{lemma}

\begin{remark}\label{rem:flow}
Actually the deformation lemma in \cite{Lucia} is originally proved for the regular case and for $K$ positive, but it adapts in a straightforward way to the singular one, even for $K$ sign-changing.\\
Indeed, in the proof of Proposition 2.3 of \cite{Lucia} a certain deformation is used following a flow in the domain of the functional, and $I_\lambda$ decreases along that flow. In our case $I_\lambda$ is not defined in the whole Sobolev space but on $X$, but $I_\lambda(u)\to+\infty$ as $u$ approaches the boundary of $\bar X$, so that $\bar X$ is positively invariant under this flow. Hence Proposition 2.3 of \cite{Lucia} is applicable and gives Lemma \ref{lemma:deformation}.
\end{remark}

In turn, since Theorem \ref{thmcompa} implies that the functional $I_\lambda$ stays uniformly bounded on the solutions of \ref{equation}, the above Lemma can be used to show that it is possible to retract the whole space $\bar X$ onto a high sublevel $I_\lambda^b$ (see \cite[Corollary 2.8]{Mal2}, also for this issue minor changes are required).

\begin{lemma}\label{lemma:highsublevels}
Let $\lambda\notin\Lambda$ and assume \ref{H1}, \ref{H2}. If $b>0$ is sufficiently large, the sublevel $I_\lambda^b$ is a retract of $\bar X$ and hence is contractible.
\end{lemma}

\subsection{Morse inequalities for $I_\lambda$}
\

The aim of this subsection is to prove a Morse-theoretical result for $I_\lambda$, which will be crucial to get the multiplicity estimates of Theorem \ref{thm:multiplicitygeneral} and Theorem \ref{thm:multiplicity8pi16pi}.

\begin{proposition}\label{prop:genericallyMorse}
Let $\ell\in\{0,\ldots,m\}$ and let us assume $\alpha_1,\ldots,\alpha_\ell>0$.
If $\lambda\in(8\pi,+\infty)\setminus\Lambda$, then for a generic choice of the function $K$, $g$ (namely for $(K,g)$ in an open and dense subset of $\mathcal{K}_{\ell}\times \mathcal M$) there exists $b=b(K,g)>0$ such that
\begin{itemize}
\item $I_\lambda^b\equiv I^b_{\lambda,K,g}$ is a retract of $\bar X\equiv \bar X_{K,g}$,
\item any solution $u\in I^b_{\lambda,K,g}$ of \ref{equation} is nondegenerate,
\end{itemize}
where to emphasize the dependence on $K$ and $g$ we write

\[
\bar X_{K,g}=\{u\in H^1_g(\Sigma)\,:\,\int_\Sigma u \,dV_g=0,\:\:\int_\Sigma K e^{-h_m}e^u dV_g>0\}
\]
and

\[
I^b_{\lambda,K,g}=\{u\in \bar X_{K,g}\,:\,\frac12\int_\Sigma |\nabla u|^2dV_g-\lambda\log\int_\Sigma K e^{-h_m}e^u dV_g\leq b\}
\]
with $h_m$ defined in \eqref{accame}.
\end{proposition}

\begin{proof}
Let us fix $(\bar K,\bar g)\in \mathcal K_\ell\times\mathcal M$.

Next, we introduce the Banach space $\mathcal S$ of all $C^{2,\alpha}$ symmetric matrices on $\Sigma$. The set $\mathcal{M}$ of all $C^{2,\alpha}$ Riemannian metrics on $\Sigma$ is an open subset of $\mathcal S$.\\
It is easy to verify that for small $\delta>0$, and any $g\in\mathcal G_\delta:=\{g\in\mathcal S\,:\,\|g\|_{C^{2,\alpha}}<\delta\}$, $\bar g+g$ is a Riemannian metric and the sets $H^1_{\bar g+g}(\Sigma)$, $L^2_{\bar g+g}(\Sigma)$, $L^1_{\bar g+g}(\Sigma)$ coincide respectively with $H^1_{g}(\Sigma)$, $L^2_{g}(\Sigma)$, $L^1_{g}(\Sigma)$ and the two norms are equivalent.

Being $\bar K\in\mathcal K_\ell$, it satisfies \ref{H1}, \ref{H2}. Thus, it is not hard to see that for $\delta>0$ small enough $\bar K+K$ satisfies \ref{H1}, \ref{H2} for any $K\in\mathcal{H}_\delta:=\{h\in C^{2,\alpha}(\Sigma)\,:\,\|h\|_{C^{2,\alpha}(\Sigma)}<\delta\}$.

Furthermore, by Theorem \ref{thmcompa}, it suffices to take a smaller $\delta>0$ so that there exists $R>0$ such that for any $(K,g)\in\mathcal{H}_\delta\times\mathcal{G}_\delta$ all the critical points (with zero mean value) of $I_{\lambda, \bar K+K, \bar g+g}$ are contained in the ball $B_0(R)\subset H^1_{\bar g}(\Sigma)$.\\

Taking a smaller $\delta>0$, if necessary, we have by Lemma \ref{lemma:highsublevels} and Theorem \ref{thmcompa} that there exists $b>0$ such that the sublevel $I^b_{\lambda,\bar K+K,\bar g+g}$ is a retract of $\bar{X}_{\bar K+K,\bar g+g}$ for any $(K,g)\in \mathcal H_\delta\times \mathcal G_\delta$.\\
Finally, for any $u\in I^b_{\lambda,\bar K,\bar g}$ $\int_\Sigma \bar K e^{-h_m}e^u\,dV_{\bar g}\geq e^{-\frac{b}{\lambda}}$, so we can also assume that if $u\in I^b_{\lambda,\bar K,\bar g}$, then $\int_\Sigma (\bar K+K) e^{-h_m}e^u\,dV_{\bar g+g}>0$ for any $(K,g)\in \mathcal H_\delta\times \mathcal G_\delta$.

Once $\delta$ is fixed in this way it is possible to argue as in \cite{DeM2}, where a transversality Theorem, obtained in \cite{SautTemam}, is applied to deduce that the following set is an open and dense subset of $\mathcal H_\delta\times \mathcal G_\delta$

\[
\left\{
(K,g)\in\mathcal H_\delta\times \mathcal G_\delta\,:\,\begin{array}{l}\mbox{\small{any $u\in I^b_{\lambda,\bar K+K,\bar g+g}$ solution of the equation}}\\
 \mbox{\small{$-\Delta_{\bar g+g}u=\lambda\left(\frac{(\bar K+K)e^{-h_m}e^u}{\int_\Sigma (\bar K+K)e^{-h_m}e^u dV_{\bar g+g}}-\frac{1}{\int_\Sigma dV_{\bar g+g}}\right)$\;\; is nondegenerate}}
 \end{array}
\right\}.
\]
Since this holds for any choice of $(\bar K,\bar g)$ the thesis follows.

\end{proof}

As recalled in the previous subsection we do not know whether $I_\lambda$ satisfies the $(PS)$ condition or not, thus Theorem \ref{MorseIneq} can not be directly applied. However, as already pointed out in \cite{BarDemMal}, the $(PS)$-condition can be replaced by the request that appropriate deformation lemmas hold for the functional.\\
In particular a flow defined by Malchiodi in \cite{Mal2} allows to adapt to $I_\lambda$ the classical deformation lemmas \cite[Lemma 3.2 and Theorem 3.2]{Chang} needed so that Theorem \ref{MorseIneq} can be applied for $H=\bar X$ and $I=I_\lambda$. It is worth to point out that, even if the flow is defined for $K$ positive, arguing as in Remark \ref{rem:flow} it is not hard to check that the same construction applies also in the sign-changing case.\\
In conclusion the following result holds true.
\begin{proposition}\label{prop:MorseIlambda}
Let $\ell\in\{0,\ldots,m\}$ and let us assume $\alpha_1,\ldots,\alpha_\ell>0$. If $\lambda\notin\Lambda$, $a$, $b$ are regular values of $I_\lambda$ and all the critical points in $\{a\leq I_\lambda\leq b\}$ are nondegenerate, then

\[
\#\{\textrm{\small{critical points of $I_\lambda$ in $\{a\leq I_\lambda\leq b\}$}}\}\geq  \sum_{q \geq 0}   \dim(H_q(I_\lambda^b,I_\lambda^a;\Z_2)).
\]
\end{proposition}

\section{On the homology groups of barycenter sets}\label{sectcomphomol}

\setcounter{equation}{0}

In this section we compute the dimension of the homology groups of some spaces of formal barycenters which have been introduced in the previous section.

Keeping the notation of Proposition \ref{prop:lowsublevelsgeneral}, we consider the space

\[
Z_{N,M}= X_N\amalg Y_M,
\]
where $X_N=\amalg_{i=1}^N B^{g_i}$ and $Y_M=\{y_1,\ldots,y_M\}$, with $g_i$ defined in \eqref{concomp}.
For $k\in \N$, $N,M\in\N\cup\{0\}$, with $N+M\geq1$, and $q\in\N\cup\{0\}$ we set
\beq\label{d_q}
d_q(k,N,M)=\dim(\HTI_q(Bar_k(Z_{N,M}));\Z_2),
\end{equation}
with the convention that $d_q(k,N,M)=0$ if $q<0$.

\begin{proposition}\label{thm:formula}
Let $k\in \N$, $N,M\in\N\cup\{0\}$, with $N+M\geq1$, and $q\in\N\cup\{0\}$, then\\
if \quad $k+1-M \leq N$,

\[
d_q(k,N,M)=\left\{
\begin{array}{ll}
\mybinom{N+M-1}{N+M-p}\mathlarger{\sum}\limits_{\tiny\begin{array}{c}a_1+\ldots+a_N=k-p+1\\a_i\geq0\end{array}}\!\!\!\!\!\!\!\!\!\!\!\!\!\!\!\!\!\!\!s_{a_1,g_1}\ldots s_{a_N,g_N} &\mbox{\footnotesize \quad if\;\, $q=2k-p\;\;(1\leq p\leq k+1)$}\\
0 & \mbox{\quad\footnotesize otherwise;}
\end{array}\right.
\]
if \quad $k+1-M \geq N$,

\[
d_q(k,N,M)=\left\{
\begin{array}{ll}
\mybinom{N+M-1}{N+M-p}\mathlarger{\sum}\limits_{\tiny\begin{array}{c}a_1+\ldots+a_N=k-p+1\\a_i\geq0\end{array}}\!\!\!\!\!\!\!\!\!\!\!\!\!\!\!\!\!\!\!s_{a_1,g_1}\ldots s_{a_N,g_N} &\mbox{\footnotesize \quad if\;\, $q=2k-p\;\;(1\leq p\leq N)$}\\
\mybinom{N+M-s}{M-s}\mathlarger{\sum}\limits_{\tiny\begin{array}{c}a_1+\ldots+a_N=k-N-s+1\\a_i\geq0\end{array}}\!\!\!\!\!\!\!\!\!\!\!\!\!\!\!\!\!\!\!s_{a_1,g_1}\ldots s_{a_N,g_N} &\mbox{\footnotesize \quad if\;\, $q=2k-N-s\;\;(1\leq s\leq M)$}\\
0 & \mbox{\quad\footnotesize otherwise;}
\end{array}\right.
\]
where $s_{a,g}=\binom{a+g-1}{g-1}$ and $g_i$ is defined in \eqref{concomp}. \\
Moreover we adopt the following convention: if $N=0$

\[
\mathlarger{\sum}\limits_{\tiny\begin{array}{c}a_1+\ldots+a_N=h\\a_i\geq0\end{array}}\!\!\!\!\!\!\!\!\!\!\!\!\!s_{a_1,g_1}\ldots s_{a_N,g_N}=\left\{\begin{array}{ll}
1\quad&\mbox{if $h=0$}\\
0&\mbox{if $h\neq 0$.}
\end{array}
\right.
\]\\
Notice that if $k+1-M=N$ the two formulas coincide.
\end{proposition}

\begin{proof}\quad

\textit{Step $1$.}\quad The thesis holds true if $k=1$ or $N=0$. If $k=1$, $Bar_1(Z_{N,M})\cong Z_{N,M}$ and so by direct computation we 
have:

\[
d_q(1,N,M)=\left\{
\begin{array}{ll}
\sum\limits_{i=1}^N g_i\,\,(=\,\sum\limits_{i=1}^N s_{1,g_i})\quad&\mbox{\footnotesize if\;\, $q=1$}\\
N+M-1&\mbox{\footnotesize if \;\, $q=0$}\\
0&\mbox{\footnotesize otherwise.}
\end{array}
\right.
\]

If $N=0$, $Bar_k(Z_{0,M})$ is the $(k-1)$-skeleton of a $(M-1)$-symplex and so the following formula holds

\[
d_q(k,0,M)
=\left\{
\begin{array}{ll}
\mybinom{M-1}{k}&\mbox{\footnotesize if\;\, $q=k-1$}\\
0&\mbox{\footnotesize otherwise,}
\end{array}
\right.
\]
where we adopt the convention that $\binom{a}{b}=0$ if $a<b$.

\textit{Step $2$.}\quad The thesis holds true if $M=0$ for any $k\geq2$, $N\geq1$: that is, 

\beq\label{formula step2}
d_q(k,N,0)=\left\{
\begin{array}{ll}
\mybinom{N-1}{N-p}\mathlarger{\sum}\limits_{\tiny\begin{array}{c}a_1+\ldots+a_N=k-p+1\\a_i\geq0\end{array}}\!\!\!\!\!\!\!\!\!\!\!\!\!\!\!\!\!\!\!s_{a_1,g_1}\ldots s_{a_N,g_N} &\mbox{\footnotesize \quad if\;\, $q=2k-p\;\;(1\leq p\leq \min\{k+1,N\})$}\\
0 & \mbox{\footnotesize\quad otherwise.}
\end{array}\right.
\end{equation}

We will demonstrate \eqref{formula step2} by induction on $N$, for any fixed $k\geq2$.\\
If $N=0$, the formula holds by Step 1. Now, assume by induction that \eqref{formula step2} holds true for a certain $N$ and let us show its validity for $N+1$. Being

\[
X_{N+1}=X_N\amalg B^{g_{N+1}},
\]
by Proposition \ref{prop:AKN}, \eqref{formula homology smash} and \eqref{formula homology suspension} we get

\beq\label{2 step 2}
\begin{split}
&d_q(k,N+1,0)=d_q(k,N,0)+d_{q-1}(k-1,N,0)+\dim(\HTI_q(Bar_k(B^{g_{N+1}})))+\dim(\HTI_{q-1}(Bar_{k-1}(B^{g_{N+1}})))\\
&+\sum_{\ell=1}^{k-1}\dim(\HTI_q(Bar_{k-\ell}(X_N)\ast Bar_\ell(B^{g_{N+1}})))+\sum_{\ell=2}^{k-1}\dim(\HTI_q(\Sigma Bar_{k-\ell}(X_N)\ast Bar_{\ell-1}(B^{g_{N+1}}))),
\end{split}
\end{equation}
where the homology groups are intended with coefficient in $\Z_2$.
Let us compute all the terms in \eqref{2 step 2}.\\
The first two can be obtained using the inductive assumption. Next, again by the computations in \cite[Proposition 3.2]{BarDemMal}, we know that

\beq\label{3 step 2}
\dim(\HTI_q(Bar_k(B^{g_{N+1}})))=\left\{
\begin{array}{ll}
s_{a_{N+1},g_{N+1}}&\mbox{\footnotesize \quad if\;\, $q=2k-1$}\\
0&\mbox{\footnotesize\quad otherwise, }
\end{array}
\right.
\end{equation}
and so

\beq\label{4 step 2}
\dim(\HTI_{q-1}(Bar_{k-1}(B^{g_{N+1}})))=\left\{
\begin{array}{ll}
s_{a_{N+1},g_{N+1}}&\mbox{\footnotesize \quad if\;\, $q=2k-2$}\\
0&\mbox{\footnotesize\quad otherwise. }
\end{array}
\right.
\eeq
Moreover, by \eqref{formula homology join}, using \eqref{3 step 2} and the inductive assumption we have that

\beq
\label{5 step 2}
\begin{split}
&\sum_{\ell=1}^{k-1}\dim(\HTI_q(Bar_{k-\ell}(X_N)\ast Bar_\ell(B^{g_{N+1}})))=\\
&\qquad\qquad
=\left\{
\begin{array}{ll}
\mybinom{N-1}{N-p}\!\!\!\!\mathlarger{\sum}\limits_{\tiny\begin{array}{c}a_1+\ldots+a_N+\ell=k-p+1\\a_i\geq0,\;\ell\geq 1\end{array}}\!\!\!\!\!\!\!\!\!\!\!\!\!\!\!\!\!\!\!s_{a_1,g_1}\ldots s_{a_N,g_N}\,s_{a_{N+1},\ell} &\mbox{\footnotesize\quad  if\;\, $q=2k-p\;\;(1\leq p\leq N)$}\\
0 & \mbox{\footnotesize\quad otherwise,}
\end{array}\right.
\end{split}
\eeq
and

\beq
\label{6 step 2}
\begin{split}
&\sum_{\ell=2}^{k-1}\dim(\HTI_q(\Sigma Bar_{k-\ell}(X_N)\ast Bar_{\ell-1}(B^{g_{N+1}})))=\\
&\qquad\qquad
=\left\{
\begin{array}{ll}
\mybinom{N-1}{N-p+1}\!\!\!\!\mathlarger{\sum}\limits_{\tiny\begin{array}{c}a_1+\ldots+a_N+\ell=k-p+1\\a_i\geq0,\;\ell\geq2\end{array}}\!\!\!\!\!\!\!\!\!\!\!\!\!\!\!\!\!\!\!s_{a_1,g_1}\ldots s_{a_N,g_N}\,s_{a_{N+1},\ell-1}  &\mbox{\footnotesize if\;\, $q=2k-p\;\;(2\leq p\leq N+1)$}\\
0 & \mbox{\footnotesize otherwise.}
\end{array}\right.
\end{split}
\eeq
In conclusion, combining \eqref{2 step 2}, \eqref{3 step 2}, \eqref{4 step 2}, \eqref{5 step 2} and \eqref{6 step 2} we obtain that

\[
d_q(k,N+1,0)=\left\{
\begin{array}{ll}
\mybinom{N}{N+1-p}\mathlarger{\sum}\limits_{\tiny\begin{array}{c}a_1+\ldots+a_{N+1}=k-p+1\\a_i\geq0\end{array}}\!\!\!\!\!\!\!\!\!\!\!\!\!\!\!\!\!\!\!s_{a_1,g_1}\ldots s_{a_{N+1},g_{N+1}} &\mbox{\footnotesize \quad if\;\, $q=2k-p\;\;(1\leq p\leq \min\{k+1,N+1\})$}\\
0 & \mbox{\footnotesize\quad otherwise,}
\end{array}\right.
\]
so \eqref{formula step2} holds true for $N+1$ and this completes the proof of \eqref{formula step2}.

\textit{Step $3$.}\quad Conclusion.\\

We will prove the formula by induction on $M$, with $k\geq2$ and $N\geq1$ fixed.\\
If $M=0$ the thesis is true by \textit{Step $2$}. Now, let us suppose that \eqref{formula step2} holds for $M$ and we prove that then it is also true for $M+1$.\\
Being

\[
Z_{N,M+1}=Z_{N,M}\amalg\{y_{M+1}\},
\]
and $\HTI_*(Bar_k(\{y_{M+1}\}))=0$, by \eqref{formula homology smash} and \eqref{formula homology suspension} we get

\[
d_q(k,N,M+1)=d_q(k,N,M)+d_{q-1}(k-1,N,M).
\]
Hence by the inductive assumption we can compute $d_q(k,N,M+1)$, obtaining
that

\bigskip 
if \quad $k+1-(M+1)  \leq N$

\[
d_q(k,N,M+1)=\left\{
\begin{array}{ll}
\mybinom{N+(M+1)-1}{N+(M+1)-p}\mathlarger{\sum}\limits_{\tiny\begin{array}{c}a_1+\ldots+a_N=k-p+1\\a_i\geq0\end{array}}\!\!\!\!\!\!\!\!\!\!\!\!\!\!\!\!\!\!\!s_{a_1,g_1}\ldots s_{a_N,g_N} &\mbox{\footnotesize \quad if\;\, $q=2k-p\;\;(1\leq p\leq k+1)$}\\
0 & \mbox{\quad\footnotesize otherwise;}
\end{array}\right.
\]
if \quad $k+1-(M+1) \geq N$

\[
d_q(k,N,M+1)=\left\{
\begin{array}{ll}
\mybinom{N+(M+1)-1}{N+(M+1)-p}\mathlarger{\sum}\limits_{\tiny\begin{array}{c}a_1+\ldots+a_N=k-p+1\\a_i\geq0\end{array}}\!\!\!\!\!\!\!\!\!\!\!\!\!\!\!\!\!\!\!s_{a_1,g_1}\ldots s_{a_N,g_N} &\mbox{\footnotesize \quad if\;\, $q=2k-p\;\;(1\leq p\leq N)$}\\
\mybinom{N+(M+1)-s}{(M+1)-s}\mathlarger{\sum}\limits_{\tiny\begin{array}{c}a_1+\ldots+a_N=k-N-s+1\\a_i\geq0\end{array}}\!\!\!\!\!\!\!\!\!\!\!\!\!\!\!\!\!\!\!s_{a_1,g_1}\ldots s_{a_N,g_N} &\mbox{\footnotesize \quad if\;\, $q=2k-N-s\;\;(1\leq s\leq M+1)$}\\
0 & \mbox{\quad\footnotesize otherwise.}
\end{array}\right.
\]
So the formula holds for $M+1$ and this concludes the proof.
\end{proof}

\begin{lemma}\label{lemma:formula}
Let $N$, $M\in\N\cup\{0\}$, $N+M\geq1$ and let $W_{N,M,J_{\lambda}}$ be the set defined in \eqref{WNMJ}, then

\[
\dim(\HTI_q(Bar_1(W_{N,M,J_\lambda});\Z_2))=
\left\{
\begin{array}{ll}
N+M-1&{q=0}\\
\sum_{i=1}^N g_i+|J_\lambda|&{q=1}\\
0&\mbox{otherwise.}
\end{array}
\right.
\]
\end{lemma}
\begin{proof}
Being $Bar_1(W_{N,M,J_\lambda})\cong W_{N,M,J_\lambda}$ it is immediate to see that

\[
\dim(\HTI_q(W_{N,M,J_\lambda};\Z_2))=
\left\{
\begin{array}{ll}
N+M-1&{q=0}\\
\sum_{i=1}^N (g_i+|J_{\lambda,A_i}|)+\sum_{h=1}^{M
} |J_{\lambda,C_h}|&{q=1}\\
0&\mbox{otherwise,}
\end{array}
\right.
\]
hence the thesis follows observing that

\[
\sum_{i=1}^N |J_{\lambda,A_i}|+\sum_{h=1}^{M} |J_{\lambda,C_h}|=|J_\lambda|,
\]
where $J_{\lambda,A_i}$ and $J_{\lambda,C_h}$ are defined in \eqref{JlambdaAC}.
\end{proof}

\section{Conclusion of the proofs of the main results}\label{sectproof}

\setcounter{equation}{0}

In order to prove our main results the following two Propositions will be of use.

\begin{proposition}\label{prop:nuova}
Let $\lambda\in(8k\pi,8(k+1)\pi)$, $k\in\N$, and assume \ref{H1}, \ref{H2}.\\
If $b>0$ is such that $I_\lambda^b$ is contractible and $\Sigma^+\simeq Z_{N,M}$, where $Z_{N,M}$ is defined in \eqref{ZNM}, then there exists $L>0$ sufficiently large so that

\[
\dim(H_{q+1}(I_\lambda^b,I_\lambda^{-L};\Z_2))\geq\dim(\tilde H_q(Bar_k(Z_{N,M});\Z_2))\qquad\mbox{for any $q\geq 0$.}
\]
\end{proposition}

\begin{proof}
By assumption $I_\lambda^b$ is contractible thus, from the exactness of the homology sequence,

\[
\ldots\to\tilde{H}_{q+1}(I_\lambda^{-L};\Z_2)\to\tilde{H}_{q+1}(I_\lambda^{b};\Z_2)\to{H}_{q+1}(I_\lambda^{b},I_\lambda^{-L};\Z_2)\to\tilde{H}_{q}(I_\lambda^{-L};\Z_2)\to\ldots
\]
we derive that

\[
\begin{array}{l}
{H}_{q+1}(I_\lambda^{b},I_\lambda^{-L};\Z_2)\cong\tilde{H}_q(I_\lambda^{-L};\Z_2),\qquad\quad \mbox{for any $q\geq0$},\\
H_0(I_\lambda^{b},I_\lambda^{-L};\Z_2)=0.
\end{array}
\]
Let us consider the continuous projection $\Psi$ introduced in Proposition \ref{prop:lowsublevelsgeneral} and the map

\[
\begin{array}{rccl}
j:& Bar_k(Z_{N,M})&\longrightarrow& I_\lambda^{-L}\\
&\sigma=\sum_{i=1}^k t_i\delta_{x_i}&\mapsto&\varphi_{\mu,\sigma},
\end{array}
\]
which is well defined by Lemma \ref{lemma:testfunctionsgeneral} $\emph{(i)}$ applied with $Z=Z_{N,M}$.\\
Then $\Psi\circ j$ is homotopically equivalent to the identity on $Bar_k(Z_{N,M})$. This fact follows from Proposition \ref{prop:lowsublevelsgeneral} and Lemma \ref{lemma:testfunctionsgeneral} $\emph{(ii)}$.\\
Hence, $\Psi_*\circ j_*=\Id_{| H_*(Bar_k(Z_{N,M}))}$ and so the desired conclusion follows by

\[
\dim(\tilde{H}_q(I_\lambda^{-L};\Z_2))\geq\dim(\tilde{H}_q(Bar_k(Z_{N,M});\Z_2)).
\]
\end{proof}

\begin{proposition}\label{prop:nuova2}
Let $\lambda\in(8\pi,16\pi)$ and assume \ref{H1}, \ref{H2}.\\
If $b>0$ is such that $I_\lambda^b$ is contractible and $\Sigma^+\simeq Z_{N,M}$, where $Z_{N,M}$ is defined in \eqref{ZNM}, then there exists $L>0$ sufficiently large so that

\[
\dim(H_{q+1}(I_\lambda^b,I_\lambda^{-L};\Z_2))\geq\dim(\tilde H_q(Bar_1(W_{N,M,J_\lambda});\Z_2))\qquad\mbox{for any $q\geq 0$,}
\]
where $W_{N,M,J_\lambda}$ is defined in \eqref{WNMJ}.
\end{proposition}
\begin{proof}
The proof is completely analogous to the one of the previous proposition, where  $W_{N,M,J_\lambda}$ and $\varphi_{\mu,p,\alpha}$ play the role of $Z_{N,M}$ and $\varphi_{\mu,\sigma}$ respectively, while Proposition \ref{prop:lowsublevels8pi16pi} and Lemma \ref{lemma:testfunctions8pi16pi} must be applied instead of Proposition \ref{prop:lowsublevelsgeneral} and Lemma \ref{lemma:testfunctionsgeneral}.
\end{proof}

\begin{proof}[Proof of Theorem \ref{thm:existencegeneral}]
By Lemma \ref{lemma:highsublevels} there exists $b>0$ so that the sublevel $I_\lambda^b$ is contractible, then we are in position to apply Proposition \ref{prop:nuova} and so for $L>0$ sufficiently large

\[
\dim(H_{q+1}(I_\lambda^b,I_\lambda^{-L};\Z_2))\geq\dim(\tilde H_q(Bar_k(Z_{N,M});\Z_2))\qquad\mbox{for any $q\geq 0$.}
\]
Hence, by virtue of \ref{H3}, $\dim(H_{q+1}(I_\lambda^b,I_\lambda^{-L};\Z_2))>0$ for some $q$, as it can be directly checked applying Theorem \ref{thm:formula} and recalling that $N^+=N+M$. Therefore, $I_{\lambda}^{L}$ is not a retract of $I_{\lambda}^b$, so the conclusion follows from Lemma~\ref{lemma:deformation}.

\end{proof}

\begin{proof}[Proof of Theorem \ref{thm:existence8pi16pi}]
We argue as in the proof of Theorem \ref{thm:existencegeneral}, indeed applying in this case Lemma \ref{lemma:highsublevels} and Proposition \ref{prop:nuova2} we have that for $b$ and $L$ sufficiently large positive

\[
\dim(H_2(I^b_\lambda, I^{-L}_\lambda;\Z_2))\geq\dim(H_1(Bar_1(W_{N,M,J_\lambda});\Z_2)).
\]
Lemma \ref{lemma:formula} combined with \ref{H4} allows to see that $\dim(H_2(I^b_\lambda, I^{-L}_\lambda;\Z_2))>0$, so we conclude again by Lemma~\ref{lemma:deformation}.
\end{proof}

\begin{proof}[Proof of Theorem \ref{thm:multiplicitygeneral}]
By virtue of Proposition \ref{prop:genericallyMorse}, we can fix $(K,g)\in\mathcal K_\ell\times \mathcal M$, such that any solution $u\in I_\lambda^b$ of \ref{equation} is nondegenerate. Next, combining Proposition \ref{prop:MorseIlambda} (with $a=-L$) and Proposition \ref{prop:nuova}, we get that for $L>0$ sufficiently large

\begin{eqnarray*}
\#\{\textrm{\small{solutions to \ref{equation}}}\}&\geq&
\sum_{q\geq0}\#\{\textrm{\small{critical points in $\{-L\leq I_\lambda\leq b\}$ with index $q$}}\}\\
&\geq& \sum_{q\geq0}\dim(\tilde H_q(Bar_k(Z_{N,M});\Z_2)).
\end{eqnarray*}
We conclude by Proposition \ref{thm:formula}.
\end{proof}

\begin{proof}[Proof of Theorem \ref{thm:multiplicity8pi16pi}]
The proof is completely analogous to the one of Theorem \ref{thm:multiplicitygeneral}, where $W_{N,M,J_\lambda}$ plays the role of $Z_{N,M}$ and Proposition \ref{prop:nuova2} and Lemma \ref{lemma:formula} are applied in place of Proposition \ref{prop:nuova} and Proposition \ref{thm:formula}.
\end{proof}

\subsection*{Acknowledgements}
F. D. M. has been supported by PRIN $201274$FYK7$\_005$ and Fondi Avvio alla Ricerca - Sapienza 2015, whereas R .L.-S. and D. R. have been supported by the Feder-Mineco Grant MTM2015-68210-P and by J. Andalucia (FQM116). During the preparation of this work R. L.-S. was hosted by University of Rome La Sapienza, and he wishes to thank this institution for the kind hospitality and F .D.M. for the invitation. The authors are grateful to W. Chen, C. Li and S. Kallel for their suggestions and discussions concerning the subject. Finally, they want to express their gratitude to the referees for their careful reading and their valuable comments on the manuscript.

\end{document}